\documentclass[12pt]{amsart}
\usepackage{dsfont}
\usepackage{amssymb}
\usepackage{amsrefs}
\usepackage{mathtools}
\usepackage{amsmath}
\usepackage{amsfonts}
\RequirePackage[usenames,dvipsnames]{xcolor}
\usepackage{fullpage}
\usepackage[unicode,colorlinks=true,linkcolor=Red,citecolor=Green]{hyperref}
\usepackage{mathrsfs}
\usepackage{graphicx}
\usepackage{subfig}
\usepackage{float}
\usepackage[shortlabels]{enumitem}

\newcommand{\lrang}[1]{\left\langle #1 \right\rangle}

\theoremstyle{plain}
\newtheorem{theorem}{Theorem}
\newtheorem{proposition}[theorem]{Proposition}
\newtheorem{lemma}[theorem]{Lemma}
\newtheorem{corollary}[theorem]{Corollary}

\theoremstyle{definition}

\theoremstyle{remark}

\newcommand{\C}{\mathbb{C}}

\newcommand{\N}{\mathbb{N}}

\newcommand{\R}{\mathbb{R}}

\newcommand{\bT}{\mathbb{T}}

\newcommand{\Z}{\mathbb{Z}}

\newcommand{\cH}{\mathcal{H}}

\newcommand{\cM}{\mathcal{M}}

\newcommand{\rS}{\mathscr{S}}

\newcommand{\op}{\text{Op}}

\renewcommand{\phi}{\varphi}

\DeclareMathOperator{\Tr}{Tr}
\newcommand{\Sp}{\text{Sp}}
\newcommand{\SL}{\text{SL}}

\renewcommand{\phi}{\varphi}
\usepackage{xcolor}
\usepackage{soul}
\allowdisplaybreaks

\title{Bounds on eigenfunctions of quantum cat maps\vspace{-0.5em}}

\author{Elena Kim}
\address[Elena Kim]{Massachusetts Institute of Technology, Department of Mathematics, Cambridge, MA 02142, USA}
\email{elenakim@mit.edu}

\author{Robert Koirala}
\address[Robert Koirala]{University of California San Diego, Department of Mathematics, La Jolla, CA 92093, USA}
\email{rkoirala@ucsd.edu}

\thanks{We would like to thank the anonymous referees for a careful
reading of the paper and their useful comments. The authors gratefully acknowledge Semyon Dyatlov, who is  partially supported by NSF CAREER grant DMS-1749858, for suggesting and mentoring this project. The first named author is supported by NSF GRFP under Grant No. 1745302. The second named author was supported by MIT UROP funding.}
\begin{document}

\begin{abstract}
We study $\ell^\infty$ norms of $\ell^2$-normalized eigenfunctions of quantum cat maps. For maps with short quantum periods (constructed by Bonechi and de Bièvre in \cite{Bonechi-DeBievre2000_Article_ExponentialMixingAndTimeScales}) we show that there exists a sequence of eigenfunctions $u$ with $\|u\|_{\infty}\gtrsim (\log N)^{-1/2}$. For general eigenfunctions we show the upper bound $\|u\|_\infty\lesssim (\log N)^{-1/2}$. Here the semiclassical parameter is $h=(2\pi N)^{-1}$. Our upper bound is analogous to the one proved by B\'{e}rard in \cite{berard} for compact Riemannian manifolds without conjugate points.
\end{abstract}

\maketitle

\section{Introduction}
In this paper, we build on an existing body of work that examines the extreme values of eigenfunctions of classically chaotic quantum systems. Specifically, we examine the quantum cat map, one of the best-studied models.  First introduced by  Berry and Hannay in \cite{HB1980}, cat maps are the quantization of hyperbolic linear maps in $\SL(2, \Z)$ on the 2-dimensional torus.

The quantum cat map is a toy model in quantum chaos. Another
standard class of quantum chaos models is given by Laplacian eigenfunctions on negatively curved compact manifolds $(M,g)$, satisfying $-\Delta_g u=\lambda^2u$ and normalized so that $\|u\|_{L^2}=1$.
Estimating the suprema of these Laplacian eigenfunctions has been an area of great interest.  
For example, Rudnick and Sarnak in \cite{Rudnick-Sarnak} showed that on hyperbolic 3-manifolds, there exists a sequence of eigenstates $u_k$ such that $\|u_k\|_{L^\infty} \gg \lambda^{1/4}$.
In regard to upper bounds, Levitan, Avakumovi\'{c}, and H\"{o}rmander in  \cite{levitan}, \cite{avakumovic}, \cite{hormander}, respectively, proved the well-known result that on a compact Riemannian manifold, $\|u\|_{L^\infty} \lesssim \lambda^{\frac{n-1}{2}}$ where $n=\dim M$. However, it is conjectured that much stronger results hold. Specifically, in \cite{iwaniecsarnak}, Iwaniec and Sarnak  conjectured that for surfaces of constant negative curvature, $\|u\|_{L^\infty} \lesssim_\varepsilon \lambda^{\varepsilon}$ for all $\varepsilon>0$;
in the special case of Hecke eigenfunctions on arithmetic surfaces they obtained the bound $\|u\|_{L^\infty}\lesssim_\varepsilon \lambda^{\frac5{12}+\varepsilon}$. The best known bound outside of the arithmetic cases is the result of B\'erard~\cite{berard}: when $(M, g)$ has no conjugate points,  $\|u\|_{L^\infty} = O(\lambda^{\frac{n-1}{2}}/\sqrt{\log \lambda})$. 

In this paper, we are concerned with metaplectic transformations, the quantizations of hyperbolic maps
\begin{equation}
  \label{e:A-intro}
A =\begin{bmatrix} a & b \\ c & d \end{bmatrix}\in \SL(2, \Z),\quad
|a+d|>2,\quad
ab,cd\in 2\mathbb Z.
\end{equation}
We decompose $L^2(\R^n)$ into a direct integral of finite-dimensional spaces $\cH_N(\theta)$, where $\theta \in \bT^{2}$ and $N \in \N$. 
As shown in Section \ref{preliminaries}, the condition that $ab,cd$ be even is needed to make sure that the metaplectic map
associated to~$A$ descends to a map from $\cH_N(0)$ to itself. We denote the resulting quantum cat map by $M_{N,0}$. An explicit basis for each $\cH_N(\theta)$ is given by Lemma \ref{lem:basis}. For $u \in \cH_N(\theta)$, we use $\|u\|_{\ell^p}$ to denote the standard $\ell^p$ norms applied to coefficients of this basis. We prove Theorem \ref{thm:lowerbound} and Theorem \ref{thm:upperbound}, bounds on the $\ell^\infty$ norm of eigenfunctions on~$\cH_N(0)$. See Figure~\ref{fig:norms} for a numerical illustration.
Note that due to the normalization in the spaces $\ell^2$ and $\ell^\infty$, the standard bound on Laplacian eigenfunctions
$\|u\|_{L^2}\lesssim \|u\|_{L^\infty}\lesssim \lambda^{\frac{n-1}{2}}\|u\|_{L^2}$ becomes the bound $\frac{1}{\sqrt{N}}\|u\|_{\ell^2}\leq\|u\|_{\ell^\infty}\leq \|u\|_{\ell^2}$.

In \cite{Bonechi-DeBievre2000_Article_ExponentialMixingAndTimeScales}, Bonechi and De Bi{\`e}vre prove that for each $A$, there exists a sequence of $M_{N,0}$ with ``short" periods. Faure, Nonnenmacher, and Bi{\`e}vre use this result in \cite{scarred} to show that there exists a sequence of eigenfunctions that are localized, as quantified by semiclassical measures. We also utilize \cite{Bonechi-DeBievre2000_Article_ExponentialMixingAndTimeScales} to show there exists a sequence of $M_{N, 0}$ with localized eigenfunctions, as demonstrated by the following lower bound.
\begin{theorem}\label{thm:lowerbound}
Suppose $A$ is a matrix satisfying~\eqref{e:A-intro} with positive eigenvalues, even trace, and coprime off-diagonal terms. Then we can find a sequence of odd integers $N_k\to\infty$ such that for all $\varepsilon>0$, there exists $k_0$ such that for all $k\geq k_0$,  there exists an eigenfunction $u$ of $M_{N,0}$ with $\|u\|_{\ell^2}=1$ and
\begin{align}
 \|u\|_{\ell^\infty} \geq \frac{1-\varepsilon}{\sqrt{2\log_\lambda N_k}}.
\end{align}
\end{theorem}

Under more general assumptions, for all odd $N$, we have the following upper bound. 
\begin{theorem}\label{thm:upperbound}
Suppose $A$ is a matrix satisfying~\eqref{e:A-intro}. Then for $0 < \varepsilon < 1$, there exists $N_0$ such that for all odd $N \geq N_0$, if  $u$ is an eigenfunction of $M_{N, 0}$ with $\|u\|_{\ell^2}=1$ then 
\begin{align}
\|u\|_{\ell^\infty}\leq \frac{1}{\sqrt{(1-\varepsilon)\log_\lambda N}}.
\end{align}

\end{theorem}

\begin{figure}
    \centering
    \includegraphics[scale=.39]{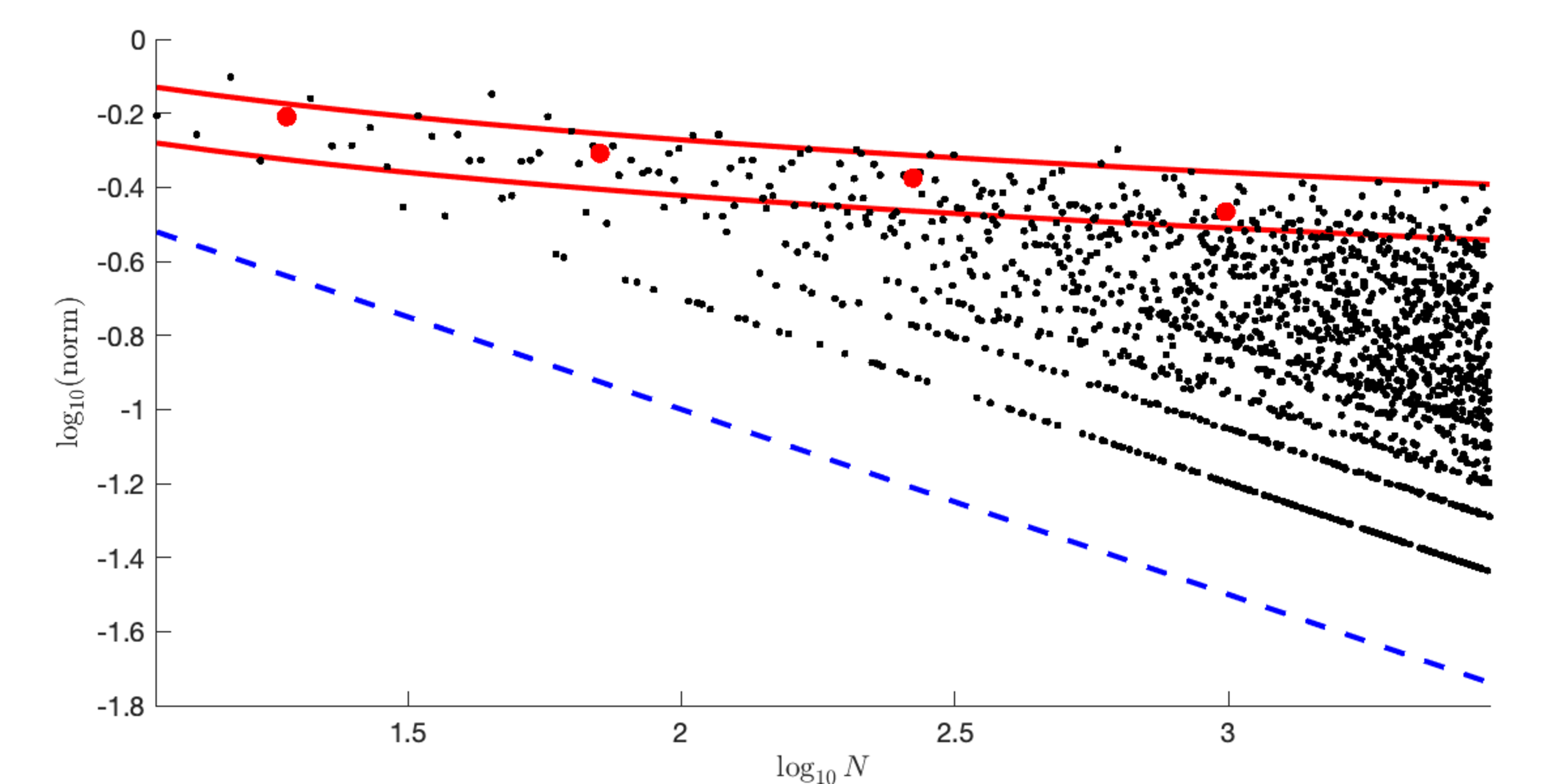}
    \caption{The plot of the maximal $\ell^\infty$-norm of an $\ell^2$-normalized eigenfunction of $M_{N,0}$ where $M_{N,0}$ is associated with $A=\begin{bmatrix} 2 & 3 \\ 1 & 2\end{bmatrix}$. The lower and upper bounds from Theorems \ref{thm:lowerbound} and \ref{thm:upperbound}, $(2\log_\lambda N)^{-1/2}$ and $(\log_\lambda N)^{-1/2}$, respectively, give the solid red lines. The dotted blue line is the trivial lower bound $N^{-1/2}$. Finally, the large red dots correspond to the sequence $N_k$ from  Theorem \ref{thm:lowerbound}.}
    \label{fig:norms}
\end{figure}

Analogous statements can be proven for a sequence of even $N$ and any $\theta$, using similar proofs to those of Theorems \ref{thm:lowerbound} and \ref{thm:upperbound}. However, we exclude these arguments as they are overly technical and do not introduce any novel ideas.

The $\ell^\infty$ bounds on eigenfunctions of quantum cat maps have been
extensively studied in arithmetic quantum chaos, see~\cite{Kurlberg-Rudnick-1,Kurlberg-1,Olofsson-1,Olofsson-2,Olofsson-3}. These works have focused on \emph{Hecke eigenfunctions},
which are joint eigenfunctions of the quantum cat map $M_{N,0}$ and the Hecke operators,
constructed in this setting by Kurlberg--Rudnick~\cite{Kurlberg-Rudnick-0}.
There always exists an orthonormal basis of $\mathcal H_N(0)$ consisting of Hecke eigenfunctions;
however, due to the possibility of large multiplicities of the eigenspaces of $M_{N,0}$ (see~\cite[footnote~3]{Kurlberg-Rudnick-0}) an upper
bound on the $\ell^\infty$ norm of Hecke eigenfunctions does not imply the same bound for general eigenfunctions. We list
below the known bounds on $\ell^2$-normalized Hecke eigenfunctions $u$:
\begin{itemize}
\item Kurlberg--Rudnick~\cite{Kurlberg-Rudnick-1} proved the upper bound $\|u\|_{\ell^\infty}\lesssim_\varepsilon N^{-\frac18+\varepsilon}$.
\item Building on~\cite{Kurlberg-Rudnick-1} (which handled roughly half of the prime values of $N$), Kurlberg~\cite{Kurlberg-1} showed that for all but finitely many \emph{prime} values of $N$ one
has the upper bound $\|u\|_{\ell^\infty}\leq 2N^{-\frac 12}$, and deduced the upper bound $\|u\|_{\ell^\infty}\lesssim_\varepsilon N^{-\frac12+\varepsilon}$ for \emph{square-free}
values of~$N$.
\item Olofsson~\cite{Olofsson-1,Olofsson-2} showed an upper bound $\|u\|_{\ell^\infty}\lesssim N^{-\frac14}$
for \emph{most} values of~$N$ (in the sense of density as $N\to\infty$). On the other hand, these
papers also construct eigenfunctions satisfying a lower bound $\|u\|_{\ell^\infty}\gtrsim N^{-\frac14}$
for most values of $N$ which are \emph{not square-free}.
\end{itemize}

\section{Preliminaries} \label{preliminaries}
We begin with a review of the necessary definitions for this paper. First, recall the semiclassical Weyl quantization. For $a \in \rS(\R^2)$ and a semiclassical parameter $h \in (0,1]$, 
$$\op_h(a)f(x) \coloneqq \frac{1}{2\pi h} \int_{\R^{2}} e^{\frac{i}{h} (x -x') \xi} a\left( \frac{x + x'}{2}, \xi \right) f(x') dx' d\xi, \quad  f \in \rS(\R).$$

Define the symbol class $$S(1) =\left\{a \in C^\infty\left(\R^{2}\right): \sup_{(x, \xi) \in \R^2} \left|\partial^\alpha_{(x, \xi)} a \right| < \infty \text{ for all } \alpha \in \N^{2}\right\},$$
which naturally induces the seminorms $\|a\|_{C^m} \coloneqq \max_{|\alpha| \leq m } \sup_{\R^{2}} |\partial_{(x, \xi)}^\alpha a |$ for $m \in \N_0$. From \cite{z12semiclassical}*{Theorem 4.16}, we know for $a \in S(1)$, $\op_h(a)$ acts on both $\mathscr{S}(\R)$ and  $\mathscr{S}'(\R)$.

Now, let $\omega=(y, \eta), z=(x, \xi) \in \R^{2}$. Define the \textit{standard symplectic form} $\sigma$ on $\R^{2}$ by $\sigma(z, \omega) \coloneqq \xi y - x \eta$ and define the \textit{quantum translation} by $U_\omega \coloneqq \op_h(a_\omega)$, where $a_\omega(z)\coloneqq \exp(\frac{i}{h} \sigma(\omega, z))$. Noting that $a_\omega(z) \in S(1)$, we see $U_\omega$ is well-defined and acts on $\mathscr{S}(\R)$. In \cite{z12semiclassical}*{Theorem 4.7}, it is shown that $$U_\omega f(x) = e^{\frac{i}{h} \eta x - \frac{i}{2h}y \eta } f(x-y).$$
Thus, $U_\omega$ is a unitary operator on $L^2(\R)$ that satisfies the following exact Egorov's theorem,
\begin{equation}\label{eq:egorov}
U^{-1}_\omega \op_h(a) U_\omega =\op_h(\tilde{a}) \quad \text{for all } a \in S(1), \quad \tilde{a}(z)\coloneqq a(z+\omega).
\end{equation}
From the fact that $U_\omega U_{\omega'} = e^{\frac{i}{2h} \sigma(\omega, \omega')} U_{\omega+\omega'}$, we deduce the following commutator formula,
\begin{equation}\label{eq:commutator}
U_\omega U_{\omega'} = e^{\frac{i}{h} \sigma(\omega, \omega')} U_{\omega'}U_\omega.
\end{equation}

Now let $\Sp(2, \R)$ be the group of real symplectic $2 \times 2$ matrices. In other words, $A \in \Sp(2, \R)$ if and only if $\sigma(Az, A\omega)=\sigma(z,\omega)$. Note that in this specific 2-dimensional case, $\Sp(2, \R)=\SL(2, \R)$.
For each $A \in \SL(2, \R)$, denote by $\cM_A$ the set of all unitary transformations $M :L^2(\R) \rightarrow L^2(\R)$ satisfying the following exact Egorov's theorem,
\begin{equation}\label{eq:MA}
M^{-1} \op_h(a) M= \op_h(a \circ A) \quad \text{for all } a \in S(1).
\end{equation}
From \cite{z12semiclassical}*{Theorem 11.9}, we have both existence of these transformations and uniqueness up to a unit factor.

Then, $\cM\coloneqq \cup_{A \in \SL (2, \R)} \cM_A$ is a subgroup of unitary transformations of $L^2(\R)$ called the \textit{metaplectic group} and the map $M \mapsto A$ is a group homomorphism $\cM \rightarrow \SL(2, \R)$. As a corollary of equation (\ref{eq:MA}), we obtain the following intertwining of the metaplectic and quantum transformations: 
$M^{-1} U_\omega M=U_{A^{-1} \omega}$ for all $M \in \cM_A$, $\omega \in \R^{2}.$

We turn our attention to quantizations of functions on the torus $\bT^{2} \coloneqq \R^{2}/ \Z^{2}$. Each $a \in C^\infty(\bT^{2})$ can be identified with a $\Z^{2}$-periodic function on $\R^{2}$. Note that any $a \in C^\infty(\bT^{2})$ is also an element of  $S(1)$, therefore its Weyl quantization $\op_h(a)$ is an operator on $L^2(\R)$.

By equation (\ref{eq:commutator}), we have the following commutation relations,
\begin{equation}\label{eq:commutation_relations}
\op_h(a) U_\omega = U_\omega \op_h(a) \quad \text{for all } a \in C^\infty(\bT^{2}), \quad \omega \in \Z^{2}.
\end{equation}
These commutation relations motivate a decomposition of  $L^2(\R)$ into a direct integral of finite dimensional spaces $\cH_N(\theta)$, where $\theta \in \bT^{2}$, such that $\op_h(a)$ descends onto these spaces. From \cite{Bouzouina-deBievre}*{Proposition 2.1}, to ensure the these spaces are nontrivial, for the rest of the paper, we assume
$$
h= (2\pi N)^{-1}\quad\text{where }N\in\mathbb N.
$$

We call $\cH_N(\theta)$ the space of \textit{quantum states}. Specifically, for each $\theta \in \bT^{2}$, set $$\cH_N(\theta)\coloneqq\left\{f \in \rS'(\R): U_\omega f=e^{2 \pi i \sigma(\theta, \omega) + N \pi i Q(\omega)} f  \text{ for all } \omega \in \Z^{2}\right\},$$ where the quadratic form $Q$ on $\R^{2}$ is defined by $Q(\omega)=y\eta$ for $\omega=(y, \eta) \in \R^{2}$. 
The following lemma gives an explicit basis for $\cH_N(\theta)$.
\begin{lemma}[\cite{dyatlov2021semiclassical}, Lemma 2.5]\label{lem:basis}
The space $\cH_N(\theta)$ is $N$-dimensional with a basis $\{e_j^\theta\}$ defined for 
 $j \in \{0, \ldots, N-1\}$ and $\theta=(\theta_x, \theta_\xi) \in \R^{2}$. In particular,
 $$e_j^\theta(x)\coloneqq\frac{1}{\sqrt{N}} \sum_{k \in \Z} e^{-2 \pi i \theta_\xi k} \delta\left(x- \frac{Nk+j-\theta_x}{N} \right).$$
\end{lemma}

We fix an inner product $\lrang{\cdot, \cdot}_\cH$ on each quantum state $\cH_N(\theta)$ by requiring $\{e_j^\theta\}$ to be an orthonormal basis. It can be shown using translation idenitities for $e_j^\theta$ (see \cite{dyatlov2021semiclassical}*{(2.35)}) that although each $\{e_j^\theta\}$ depends on the choice of the representative $\theta_x \in \R$, the inner product depends only on $\theta \in \bT^{2}$. We often denote the norm defined by this inner product by $\| \cdot\|_{\ell^2}$. Using the bases $\{e_j^\theta\}$, we can consider the spaces $\cH_N(\theta)$ as fibers of a smooth $N$ dimensional vector bundle over $\bT^{2}$, which we call $\cH_N$. 

For $u = \sum_{j=0}^{N-1} \alpha_j e_j^\theta$, we set
$$\|u\|_{\ell^p} \coloneqq \|(\alpha_0, \ldots, \alpha_{N-1})\|_{\ell^p}.$$

Fix $N \in \N$ and $a \in C^\infty(\bT^{2})$ to define the quantization 
$$\op_{N, \theta}(a)\coloneqq \op_h(a)|_{\cH_N(\theta)} : \cH_N(\theta) \rightarrow \cH_N(\theta), \quad \theta \in \bT^{2},$$ which depends smoothly on $\theta$. 
This restriction holds by definition of $\cH_N(\theta)$ and the commutation relations given in equation (\ref{eq:commutation_relations}).

We set
\begin{equation}
\label{eq:A-matrix}
A= \begin{bmatrix}
a & b \\ c & d \end{bmatrix} \in \SL(2, \Z)
\end{equation}
and choose a metaplectic operator $M \in \cM_A$. Recall that for $\omega = (y, \eta), z=(x, \xi) \in \Z^2$, $\sigma(z, \omega) = \xi y - x \eta$ and $Q(\omega) = y \eta$.  By \cite{dyatlov2021semiclassical}*{Lemma 2.9}, there exists a unique $\phi_A \in (\Z/2)^2$ such that for all $\omega\in\mathbb Z^2$, $Q(A^{-1} \omega) - Q(\omega) = \sigma(\phi_A, \omega) \mod 2\Z$.  
Using the definition of $\cH_N(\theta)$ and that fact that $M^{-1} U_\omega M=U_{A^{-1} \omega}$, we can verify that $M(\cH_N(\theta)) \subset \cH(A \theta +\frac{ N \phi_A}{2})$ for all $\theta \in \bT^{2}$.

Denote $M_{N, \theta} \coloneqq M|_{\cH_N(\theta)} : \cH_N(\theta) \rightarrow \cH_N(A \theta +\frac{N \phi_A}{2})$, which depends smoothly on $\theta \in \bT^{2}$.  We require the domain and range of $M_{N, \theta}$ to be the same, in other words, we must have 
\begin{equation}\label{eq:domainrange}
(I-A)\theta =\frac{N \phi_A}{2} \mod \Z^{2}.
\end{equation}
Thus, when $\theta =0$, condition (\ref{eq:domainrange}) is satisfied when $N$ is even or $\phi_A =0 $.

We henceforth assume that $\phi_A=0$, which gives  $\sigma(\phi_A, \omega) =0 \mod \Z/2$. Noting that for all $\omega =(y, \eta)$,
$$Q(A^{-1} \omega)-Q(\omega) = (dy -b \eta)(-cy +a \eta) - y \eta=-dcy^2 +2 bc y \eta- ba \eta^2,$$
we must have $dcy^2 + ba \eta^2= 0 \mod 2\Z$.  We conclude that $dc$ and $ab$ must be even. 

Assuming condition  (\ref{eq:domainrange}), we have the following exact Egorov's theorem for all $a \in C^\infty(\bT^{2})$,
$$M_{N, \theta}^{-1} \op_{N, \theta} (a) M_{N, \theta} =\op_{N, \theta} (a \circ A).$$

Essential to our proof of Theorem \ref{thm:upperbound} is the following  explicit formula for $M$.
\begin{lemma}\label{lem:explicitformula}
When $A$ is given by~\eqref{eq:A-matrix} and $b \neq 0$,
an element of $\mathcal M_A$ is given by
\begin{equation} \label{eq:explictformM}
M u(x) = \frac{\sqrt{N}}{\sqrt{|b|}} \int_\R e^{2 \pi N  i\Phi(x, y)} u(y) dy,
\end{equation}
where 
\begin{equation}\label{explictformphi}
\Phi(x, y) = \frac{d}{2b} x^2 - \frac{xy}{b} + \frac{a}{2b} y^2.
\end{equation}
\end{lemma}

We provide a short outline of this formula's derivation; for further details see Theorem 11.10 and its following remark in \cite{z12semiclassical}.
Set $\Lambda= \{(x, y, \xi, -\eta) : (x, \xi) = A(y, \eta)\}$ and note that $\Lambda$ is a Lagrangian submanifold of $\R^4$. Additionally, as $b \neq 0$, $\Lambda \ni (x,y, \xi, \eta) \mapsto (x, \eta)$ is surjective. Therefore, there exists a generating function $\Phi(x,y)$ such that $\Lambda=\{(x, y, \partial_x \Phi, \partial_y \Phi)\}$. 
As $\eta = -\partial_y \Phi$ and $\xi = \partial_x \Phi$, we know $\Phi(x, y)$ is given by (\ref{explictformphi}). 
We then have (\ref{eq:explictformM}), where the coefficient ensures $M$ is unitary, up to a unit factor. 

\section{Proof of Theorem \ref{thm:lowerbound}}\label{secproofoflowerbound}
First, following the presentation in \cite{Bonechi-DeBievre2000_Article_ExponentialMixingAndTimeScales}, in Section \ref{lowerboundsubsection1}, we show that for a sequence  $N_k$, $M_{N_k,0}$ has a large degenerate eigenspace. In Section \ref{lowerboundsubsection2}, we then reduce the proof of Theorem \ref{thm:lowerbound} to a linear algebra argument.

\subsection{Eigenspace of $M_{N_k, 0}$} \label{lowerboundsubsection1}

Consider $A$ given by~\eqref{eq:A-matrix} such that $b$ and $c$ are coprime and  $\Tr A$ is even and greater than 2.
Let $\lambda$ be the largest eigenvalue of $A$. Then for each $t \in \N$, 
\begin{equation} \label{eq:pdef}
A^t =p_t A-p_{t-1}I, \quad p_{t+1} = \Tr(A)p_t - p_{t-1}, \text{ where}\quad p_t=\frac{\lambda^t-\lambda^{-t}}{\lambda - \lambda^{-1}}\in\mathbb Z.
\end{equation}
Set $T_N =\min \{t : A^t =I \mod N \}$ and, for $k \in \N$, define
$$
N'_k\coloneqq \max \{N : A^k =I \mod N\}.
$$

Essential to our proof of Theorem \ref{thm:lowerbound} is the following statement from \cite{Bonechi-DeBievre2000_Article_ExponentialMixingAndTimeScales}. For the reader's convenience, we replicate their proof here. 
\begin{theorem}[\cite{Bonechi-DeBievre2000_Article_ExponentialMixingAndTimeScales}, Prop. 11]\label{thm:T_n N_k}
For each $k \in \N$, we have 
$N'_{2k}=2p_k$, $N'_{2k+1} =p_k + p_{k+1}$, and $T_{N'_k}=k$.
\end{theorem}

\begin{proof}
Using equation (\ref{eq:pdef}), we see that $N'_k$ is the greatest integer such that 
$$\begin{bmatrix} p_k a - p_{k-1} -1 & p_k b \\ p_k c & p_k d -p_{k-1} -1 \end{bmatrix} = \begin{bmatrix} 0 & 0 \\ 0 & 0 \end{bmatrix} \mod N'_k.$$
Recall that we assumed $b$ and $c$ are coprime. Therefore, $p_k=0 \mod N'_k$ and $p_{k-1} =-1 \mod N'_k$, which gives $N'_k = \gcd(p_k, p_{k-1} +1)$. We claim that for $s=0, \ldots, k-1$, 
\begin{equation} \label{eq:N_kdef}
N'_k = \gcd(p_{k-s} +p_s, p_{k-(s+1)} +p_{s+1}).
\end{equation}
We proceed by induction. 
Note that $p_0=0$ and $p_1=1$, therefore, (\ref{eq:N_kdef}) clearly holds for $s=0$. Now suppose we know (\ref{eq:N_kdef}) for some $s \geq 0$. Using (\ref{eq:pdef}) and the identity $\gcd(a, ca-b) =\gcd(a, b)$, we have
\begin{align*}
N'_k &= \gcd(\Tr A p_{k-s-1} -p_{k-s-2} + p_s, \text{ } p_{k-s-1} +p_{s+1})\\
&=\gcd(\Tr A (p_{k-s-1} + p_{s+1}) - \Tr A p_{s+1} - p_{k-s-2} +p_s, \text{ }  p_{k-s-1} + p_{s+1})\\
&=\gcd(p_{s+2} + p_{k-(s+2)}, \text{ }  p_{k -(s+1)} +p_{s+1}), 
\end{align*}
which completes the induction. We set $k=2 \ell$, $s=\ell$ in (\ref{eq:N_kdef}) to conclude
$$N'_{2 \ell} =\gcd(2 p_\ell, p_{\ell -1} + p_{\ell +1}) = \gcd(2 p_{\ell}, \Tr A p_{\ell}) = 2 p_\ell,$$
where the last equality follows from our assumption that $\Tr A$ is even. Similarly, setting $k = 2 \ell +1$ and $s=\ell$ in (\ref{eq:N_kdef}) gives $N'_{2\ell +1}=p_\ell + p_{\ell +1}$.

Now note that for each $k$, we have
\begin{equation}\label{eq:N_k1}
A^k=1 \mod N'_k,
\end{equation}
\begin{equation}\label{eq:N_k2}
A^{T_{N'_k}} =1 \mod N'_k, \quad \text{and} \quad A^{T_{N'_k}} =1 \mod N'_{T_{N'_k}}.
\end{equation}
From the definition of $T_{N'_k}$ and (\ref{eq:N_k1}), we see that $T_{N'_k} \leq k$. From the definition of $N'_k$ and (\ref{eq:N_k2}), we see that $N'_k \leq N'_{T_{N'_k}}$. As $\{N'_k\}$ is increasing, we conclude that $T_{N'_k} \geq k$. Therefore, $T_{N'_k}=k$. 
\end{proof}

Now, let $n(N)$ denote the period of $M_{N, 0}$; specifically
$$
n(N) \coloneqq \min \{t : M_{N, 0}^t = e^{i \phi} \text{ for some } \phi \in \R\}.
$$
Suppose $A_N$ is the matrix with integer entries that satisfies $A^{T_N}=1+N A_N$. From \cite{HB1980}*{(36)-(46)}, we know that $n(N)=T_N$ if $N$ is odd or if $N$ is even and $(A_N)_{12}$ and $(A_N)_{21}$ are even. Otherwise, $n(N) =2T_N$.

Using this formula for $n(N)$ and Theorem  \ref{thm:T_n N_k}, following \cite{Bonechi-DeBievre2000_Article_ExponentialMixingAndTimeScales}, we show an upper bound for $n(N)$ that depends only on $\lambda$ and $N$.

\begin{figure}%
    \centering
    \subfloat[\centering $N=991$ ]{{\includegraphics[width=11cm]{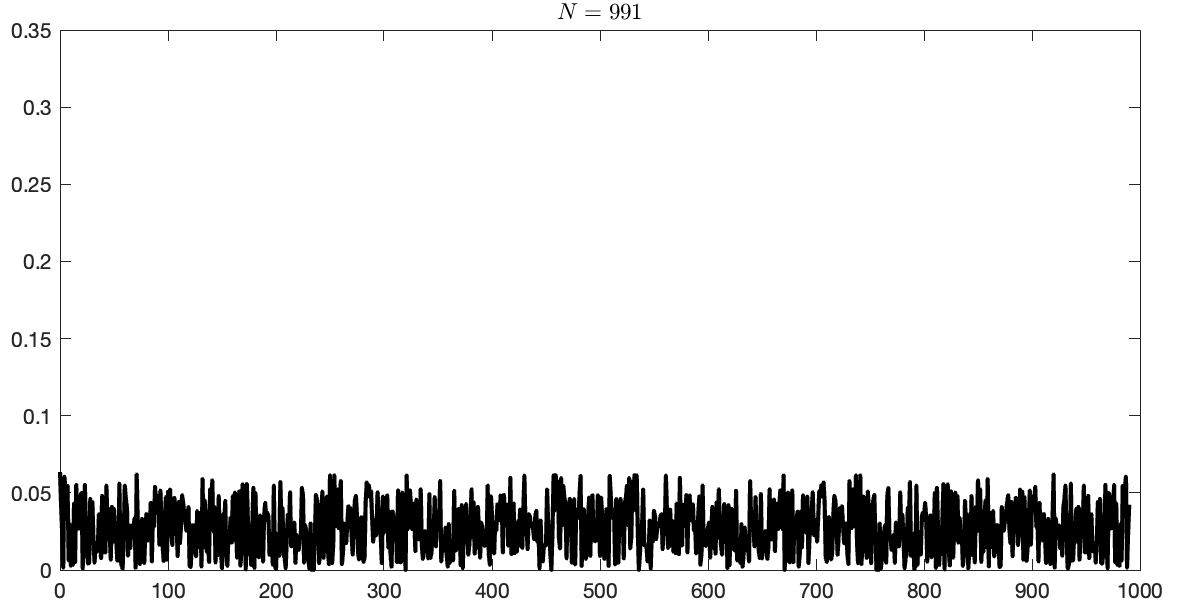} }}%
    \linebreak
    \subfloat[\centering  $N=2911$]{{\includegraphics[width=11cm]{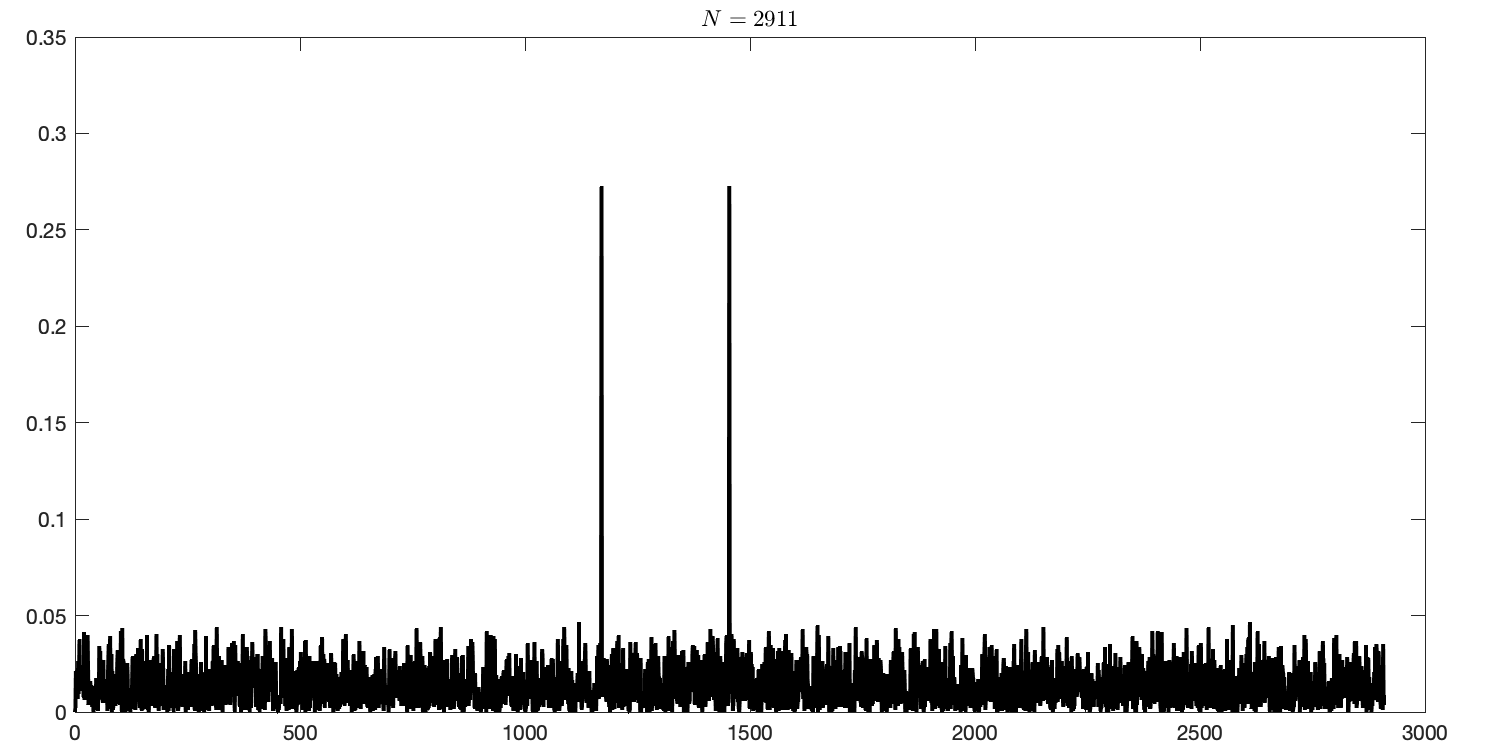} }}%
    \caption{The plots of a maximal $\ell^\infty$-norm, $\ell^2$-normalized eigenfunction of $M_{N,0}$, where $M_{N,0}$ corresponds to $A=\begin{bmatrix}
         2 & 3\\1 & 2
    \end{bmatrix}$. Specifically, each plot point corresponds to the absolute value of the $i$th coordinate of the eigenfunction for $0 \leq i \leq N-1$.
    Note that $N=2911$ is an element of the sequence $N_k$ in Corollary~\ref{lem:quantumperiod}, while $N=991$ is not.}%
    \label{fig:cateigf}%
\end{figure}

\begin{corollary}\label{lem:quantumperiod}
Let $\lambda$ be the largest eigenvalue of $A$. Then, there exists a sequence of odd $N_k$ such that $2 \log_\lambda  N_k +1 \geq n(N_k)$.
\end{corollary}

\begin{proof}
Using  our assumption that $\Tr A$ is even, we note that $p_{2k}$ is even and $p_{2k+1}$ is odd. Thus, $n(N'_{2k+1}) =T_{N'_{2k+1}}=2k+1$.  

We now formulate a bound for $n(N'_{2k+1})$ in terms of $N'_{2k+1}$. As $\Tr A >2$, we know $\lambda >1$. Additionally using Theorem \ref{thm:T_n N_k}, $N'_{2k+1} = \frac{\lambda^{k+1} + \lambda^k - \lambda^{-k-1} - \lambda^{-k}}{\lambda - \lambda^{-1}} \geq \lambda^{k}$. Therefore,  $\log_\lambda N'_{2k+1} \geq k $, which gives $2 \log_\lambda  (N'_{2k+1}) +1 \geq 2k +1 =n(N'_{2k+1})$. Labeling, $N'_{2k+1}$ as $N_{k}$, we are done.
\end{proof}

Note that the same proof ideas can be adapted for a sequence of even $N$, using the fact that $N'_{2k}$ is even. However, as the odd case is the simpler of the two, our main theorem is proven for a sequence of odd $N$.

For ease of notation,  again using $N_k \coloneqq N'_{2k+1}$, set
$$
t_k \coloneqq n(N_{k}).
$$
Therefore,   if $u$ is an eigenfunction of $M_{N_k,0}$ with eigenvalue $\lambda$,  $M_{N_{k},0}^{t_k} u=e^{i \phi_k}u=\lambda^{t_k}u$. Thus, each eigenvalue of $M_{N_k, 0}$ satisfies $\lambda^{t_k}=e^{i \phi_k}$, giving an eigenvalue whose multiplicity is at least  $\frac{N_k}{t_k}$. In other words, $M_{N_{k},0}$ has a degenerate eigenspace with dimension at least $\frac{N_{k}}{t_k}.$

\subsection{Lower Bound}\label{lowerboundsubsection2}
Now we are ready to state the lower bound in linear algebraic terms.

\begin{proposition}\label{prop:lowerbound}
Suppose $V$ is subspace of $\C^N$ with dimension at least $\frac{N}{t}$. Then,
\begin{align}
\frac{1}{\sqrt{t}}\leq\max_{u\in V,\ \|u\|_{\ell^2}=1}\|u\|_{\ell^\infty}.
\end{align}
\end{proposition}

\begin{proof}
Define $\Pi: \C^N\to V$ to be the orthogonal projection onto $V$ and let $e_j$ denote the $j$th coordinate vector. Note that
$$
\sum_{j=1}^N \|\Pi e_j\|_{\ell^2}^2=\Tr(\Pi^*\Pi)=
\Tr \Pi = \dim V \geq \frac{N}{t}.
$$
Therefore, there exists a $j$ such that $\|\Pi e_j\|_{\ell^2} \geq \frac{1}{\sqrt{t}}$.
We then have
$$\sup_{\substack{u \in V\\ \|u\|_{\ell^2}=1}}\|u\|_{\ell^\infty} =\max_j \sup_{\substack{u \in \C^N \\ \|u\|_{\ell^2}=1}} | \lrang{\Pi u, e_j} |=\max_j \sup_{\substack{u \in \C^N \\ \|u\|_{\ell^2}=1}} | \lrang{u, \Pi e_j} |=\max_j \|\Pi e_j\|_{\ell^2} \geq \frac{1}{\sqrt{t}},$$
which completes the proof.
\end{proof}

We claim that Proposition \ref{prop:lowerbound} implies Theorem \ref{thm:lowerbound}. In particular, fixing $V$ to be the degenerate eigenspace of $M_{N_k, 0}$ with dimension at least $\frac{N_k}{t_k}$ and using the fact that  $\C^{N_k} \simeq \cH_{N_k}(0)$, Proposition \ref{prop:lowerbound} implies
\begin{align}
   \frac{1}{\sqrt{2 \log_\lambda   N_{k}+1}} \leq  \frac{1}{\sqrt{t_k}} \leq\max_{u\in V,\ \|u\|_{\ell^2}=1}\|u\|_{\ell^\infty}.
\end{align}

For an explicit demonstration of this lower bound, see Figure \ref{fig:cateigf}. 

\subsection{Eigenfunctions in Proposition \ref{prop:lowerbound}}
We further examine the eigenspace $V$, which achieved the lower bound in Proposition \ref{prop:lowerbound}.
Recall we defined $\Pi$ to be the orthogonal projection onto $V$ and examined eigenfunctions of the form $\Pi (\sum_{j=0}^{N-1} \alpha_j e_j^0)$. To find a more precise formula for these eigenfunctions, we deduce a formula for $\Pi$.

Suppose $V$ corresponds to eigenvalue $\lambda$.
We know that the spectrum of $M_{N_k, 0}$ is contained in $\{z : z^{t_k} =e^{i \phi_k}\}$. Therefore, for some $s \in \Z$, $\lambda$ is of the form $e^{i \frac{\phi_k + 2\pi s}{t_k}}$. Thus, on the spectrum,
$$\frac{1}{t_k}\left(1 + \frac{z}{\lambda} + \cdots + \frac{z^{t_k-1}}{\lambda^{t_k-1}} \right) =\begin{cases}  1, & \text{ if } z = \lambda \\ 0, & \text{ else }\end{cases}.$$

We then see $$\Pi = \frac{1}{t_k} \left(I + e^{-i \frac{\phi_k + 2\pi s}{t_k}}M +  \cdots + \left(e^{-i \frac{\phi_k + 2\pi s}{t_k}}M\right)^{t_k-1} \right).$$

Therefore in our proof of the lower bound, we are actually studying the eigenfunctions of the form 
\begin{equation}\label{eq:delta}
\frac{1}{t_k} \sum_{t=0}^{t_k-1} e^{-i \frac{\phi_k + 2\pi s}{t_k}t}M^t e^0_j.
\end{equation}

\begin{figure}
   \centering
\begin{tabular}{lccccc}
& $W_{M^j g}(x)$ & $W_{M^je^0_{N/2}}(x)$ & & $W_{M^j g}(x)$ & $W_{M^je^0_{N/2}}(x)$\\
$j=0$&
\includegraphics[width=2.5cm]{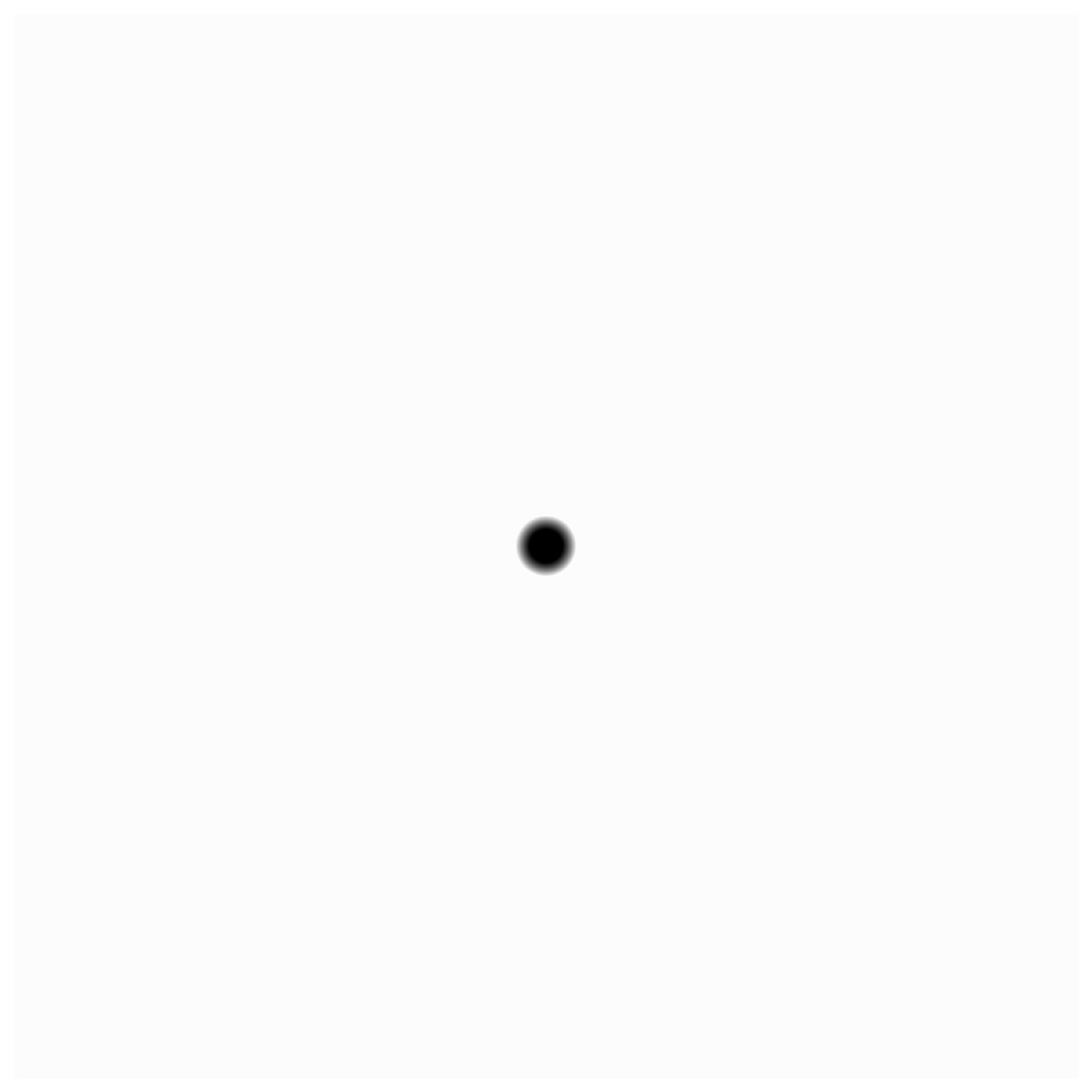}&
\includegraphics[width=2.5cm]{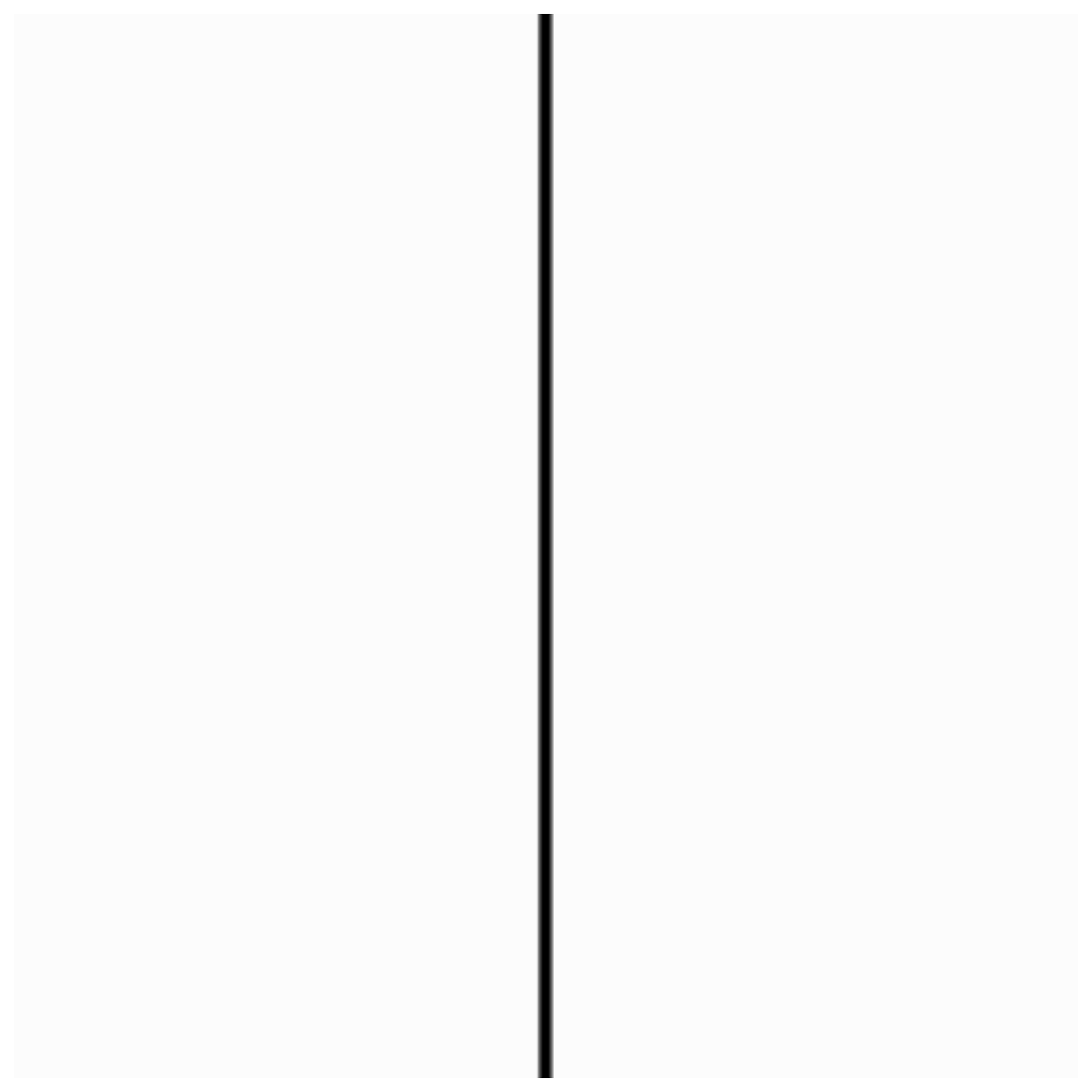} & $j=7$&
\includegraphics[width=2.5cm]{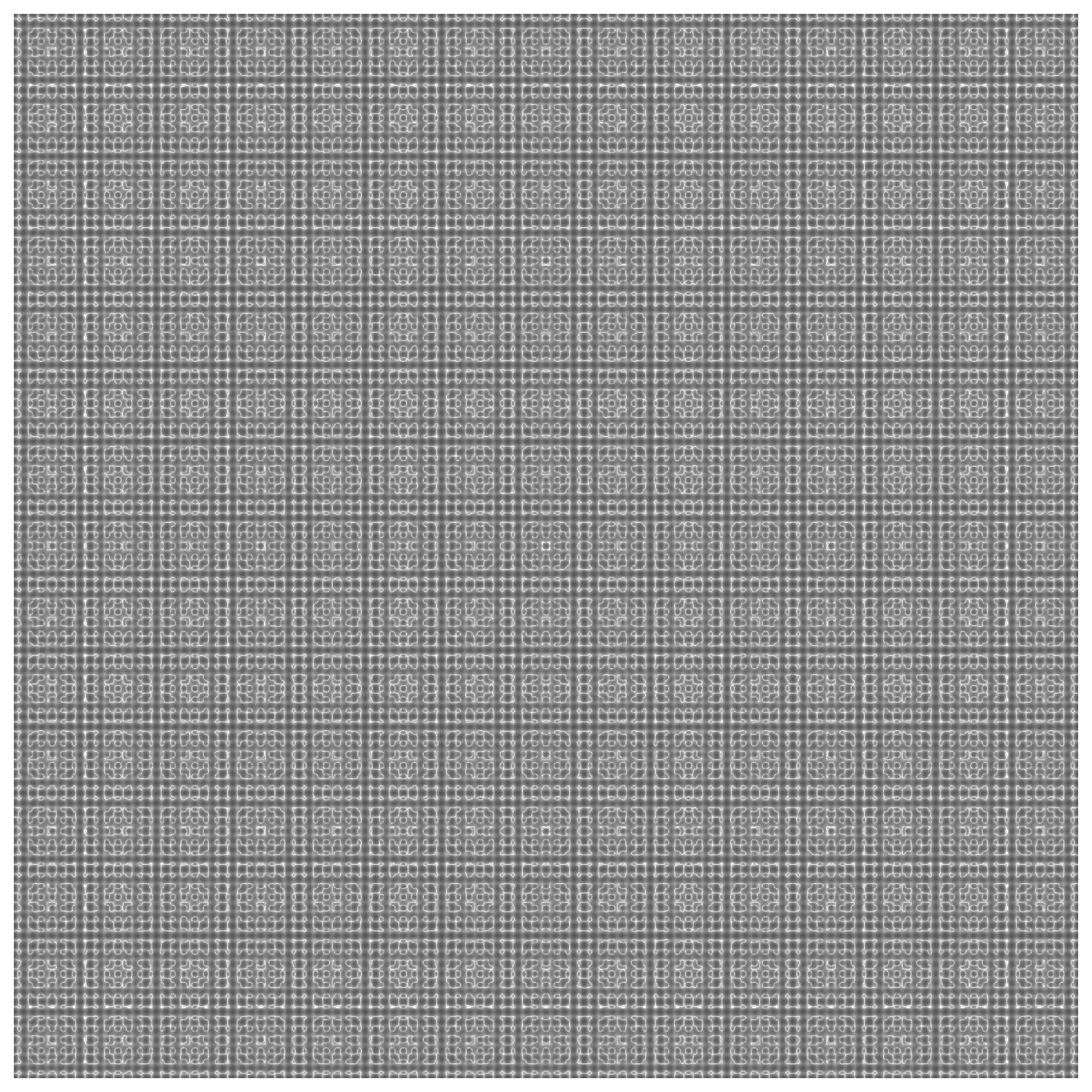}&
\includegraphics[width=2.5cm]{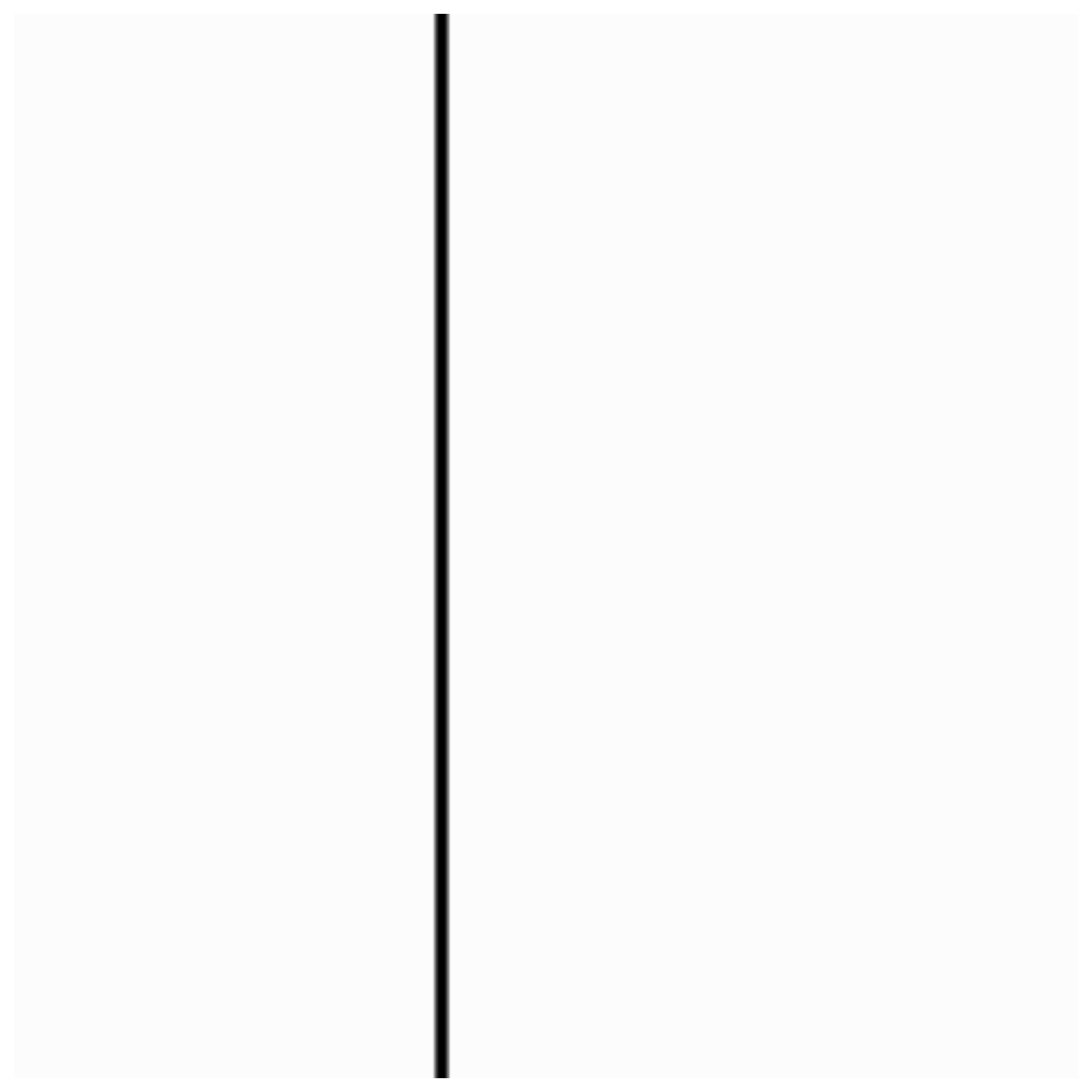} \\
$j=1$&
\includegraphics[width=2.5cm]{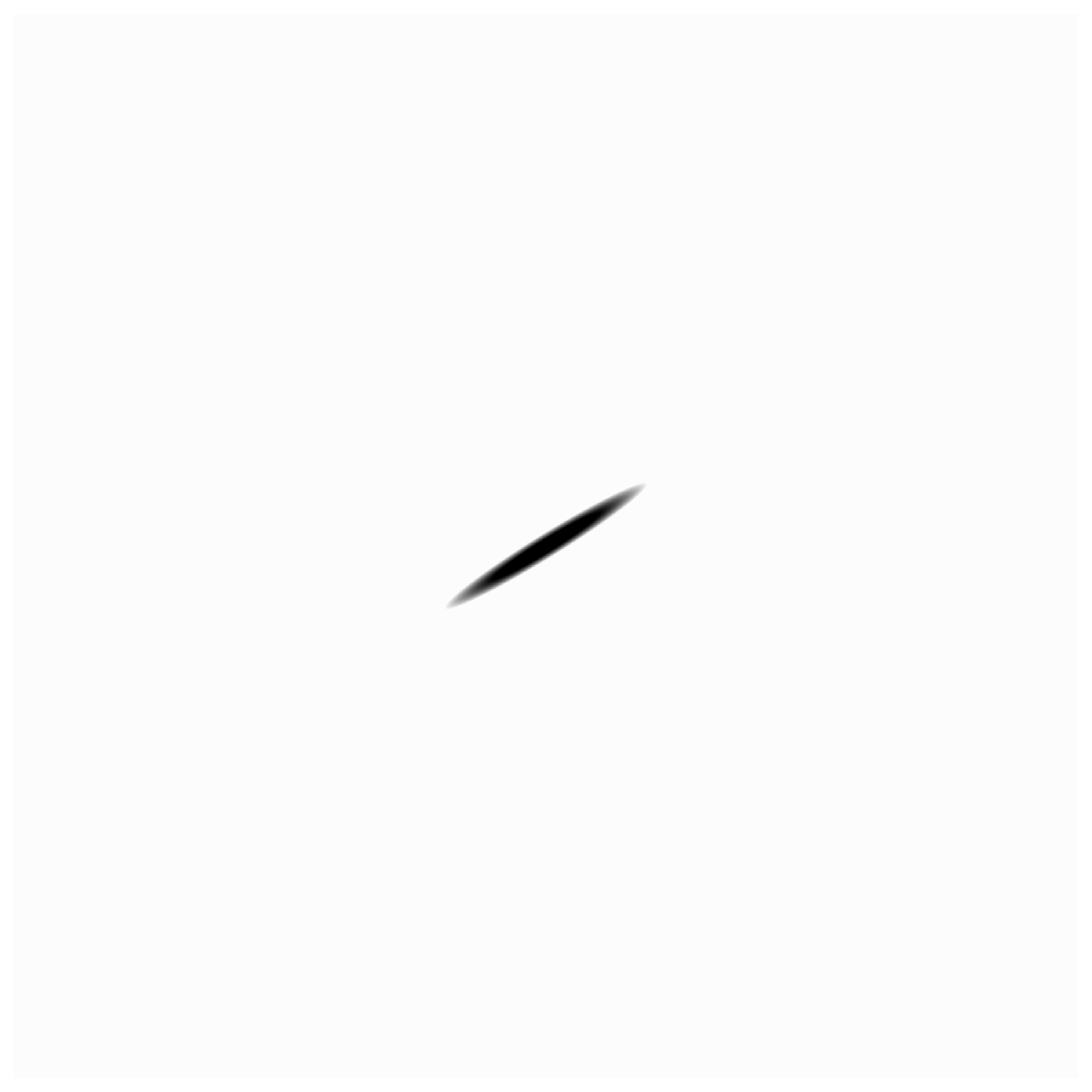}&
\includegraphics[width=2.5cm]{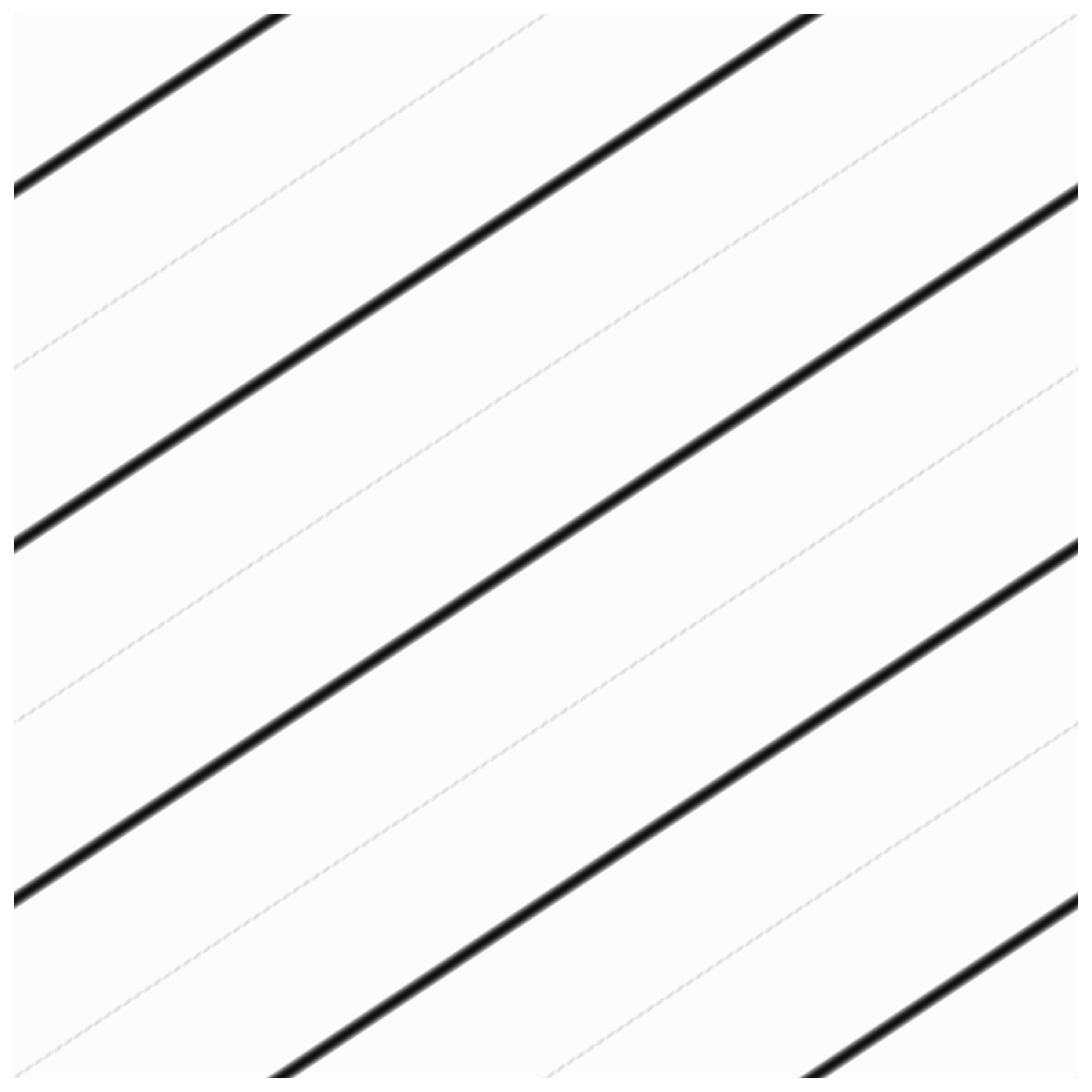} & $j=8$&
\includegraphics[width=2.5cm]{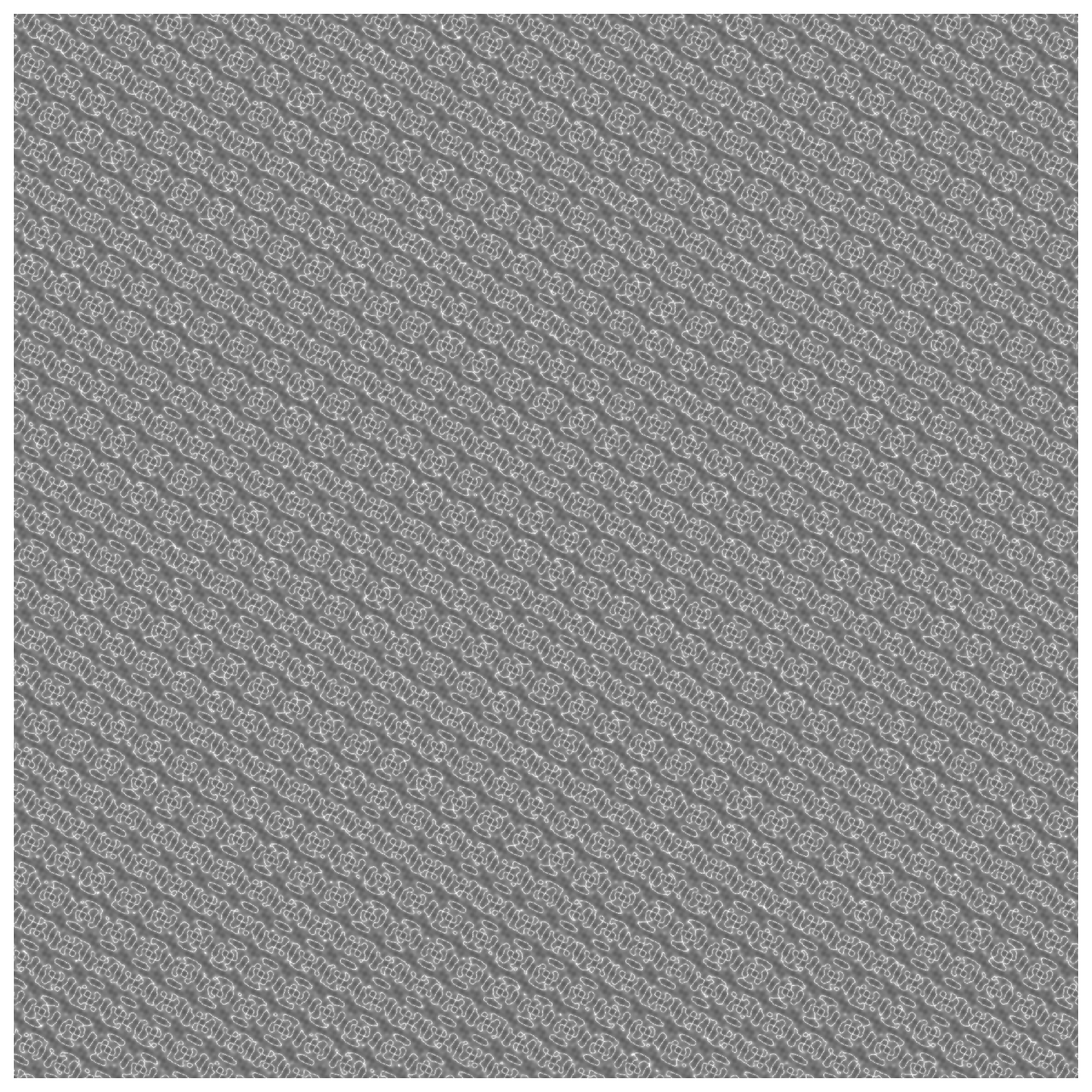}&
\includegraphics[width=2.5cm]{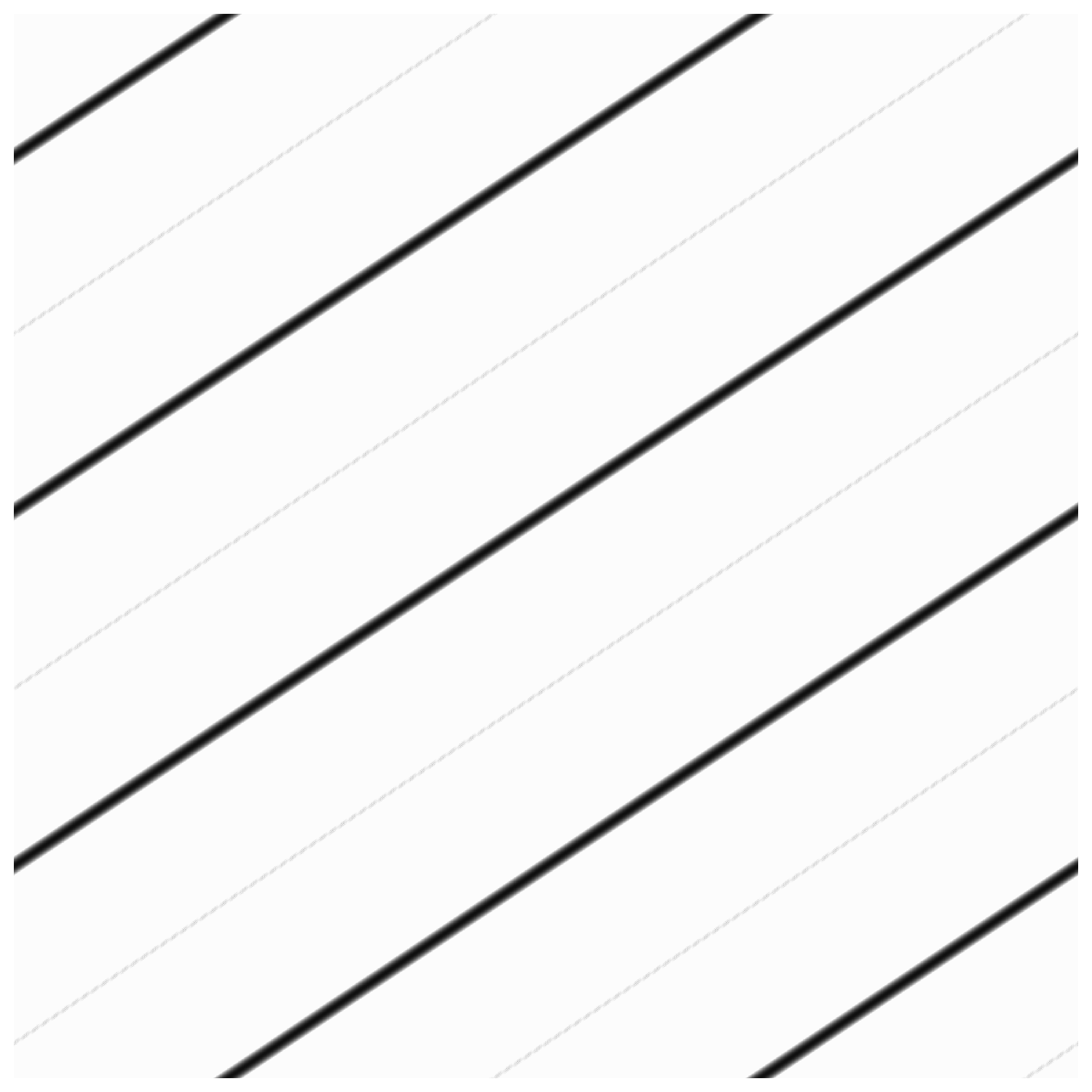} \\
$j=2$&
\includegraphics[width=2.5cm]{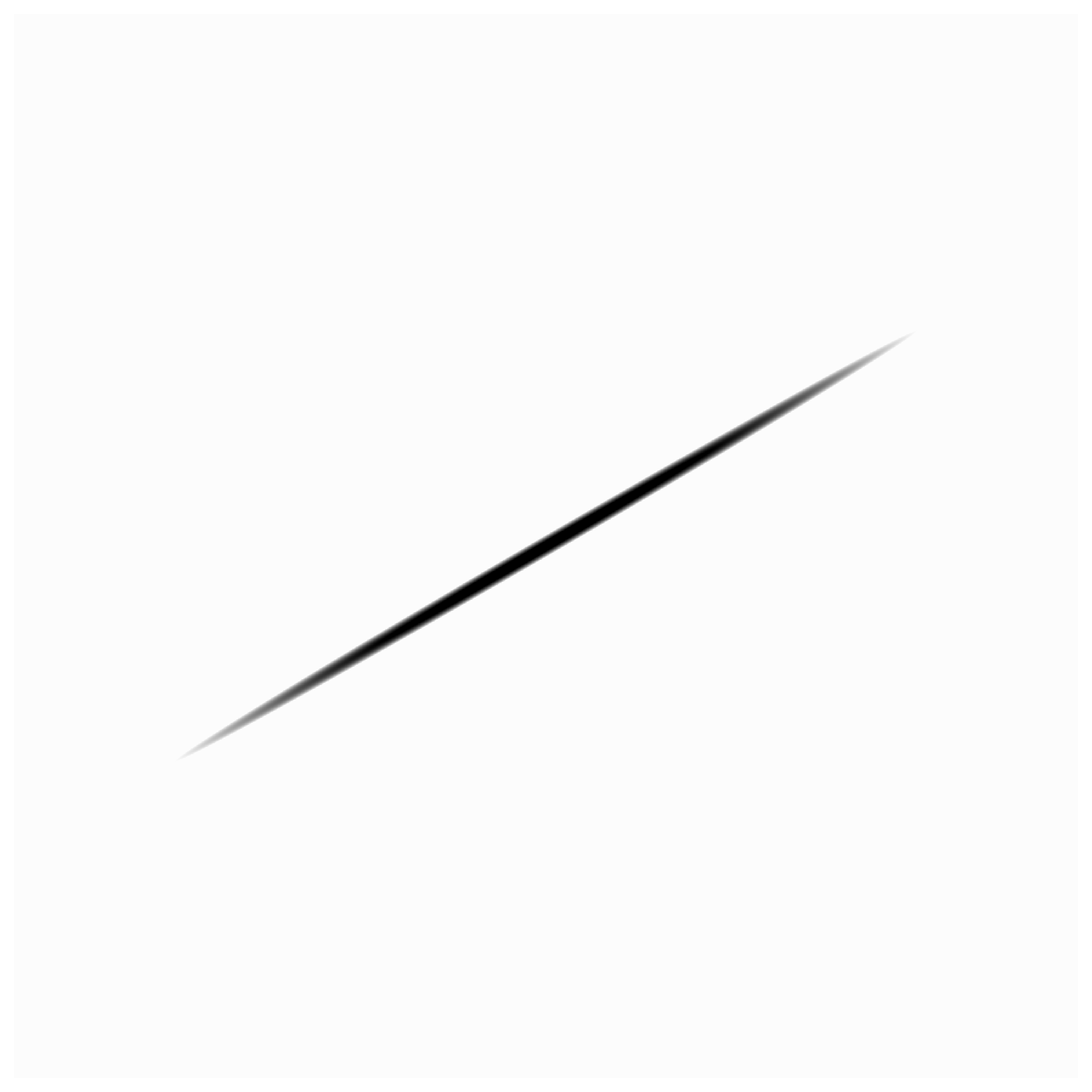}&
\includegraphics[width=2.5cm]{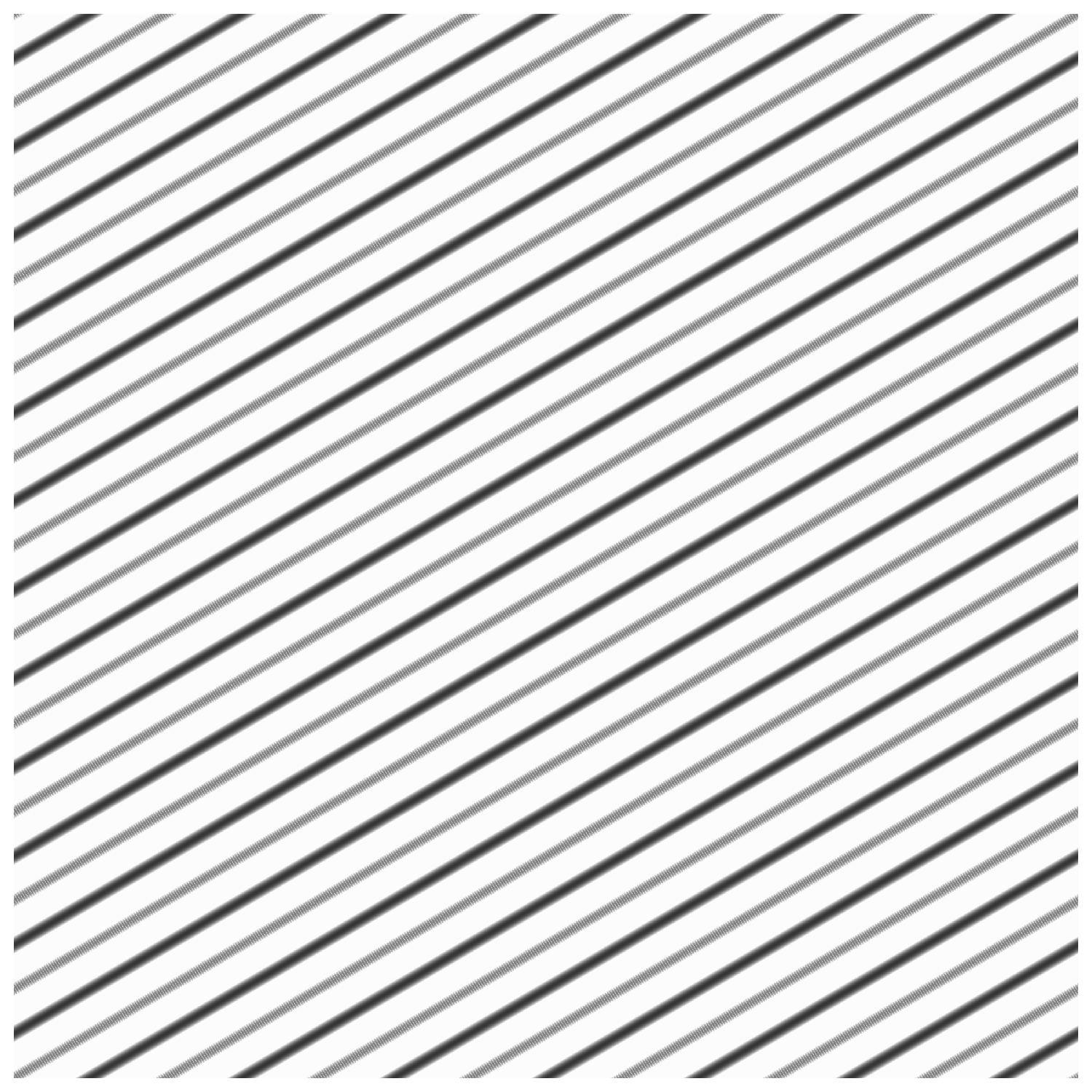} & $j=9$&
\includegraphics[width=2.5cm]{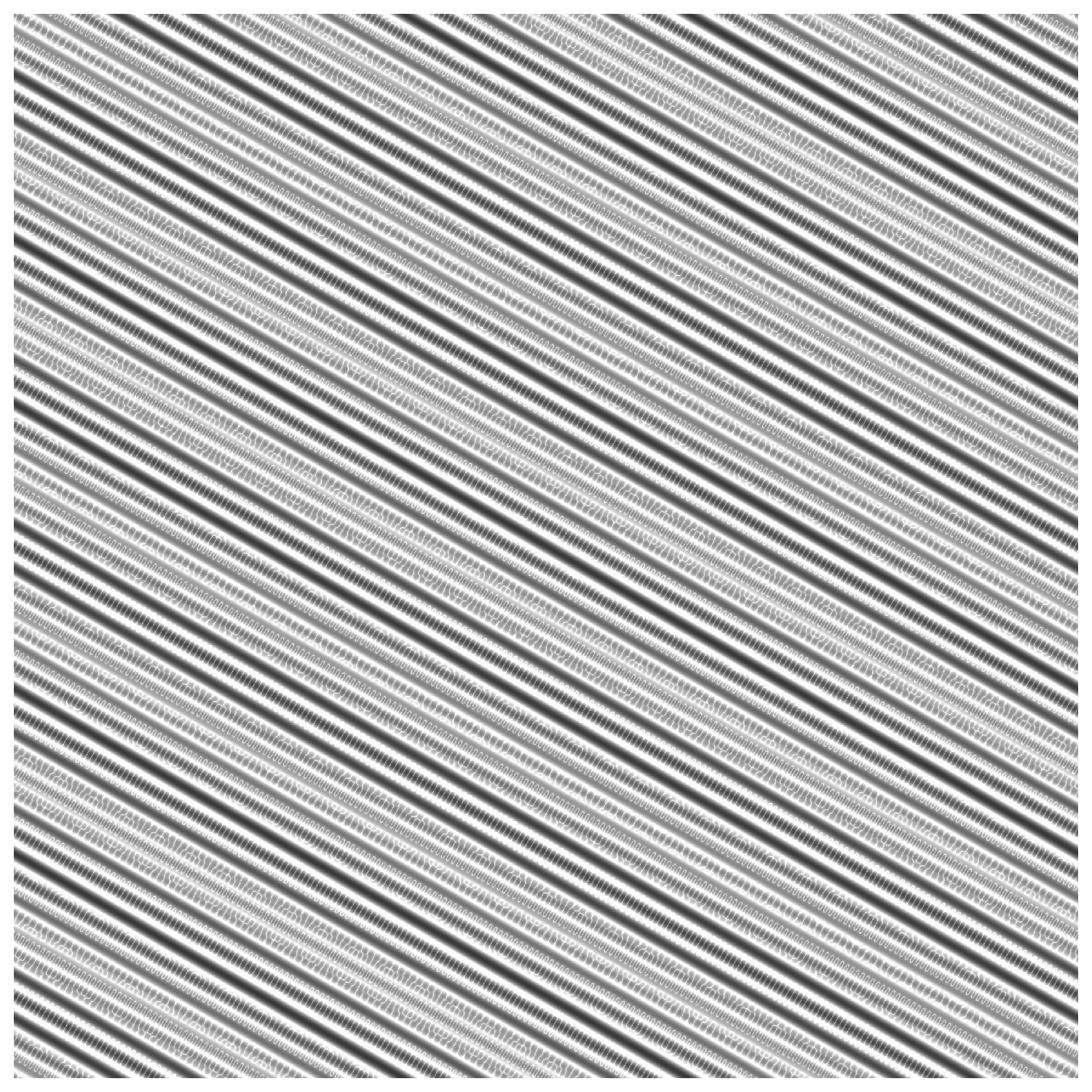}&
\includegraphics[width=2.5cm]{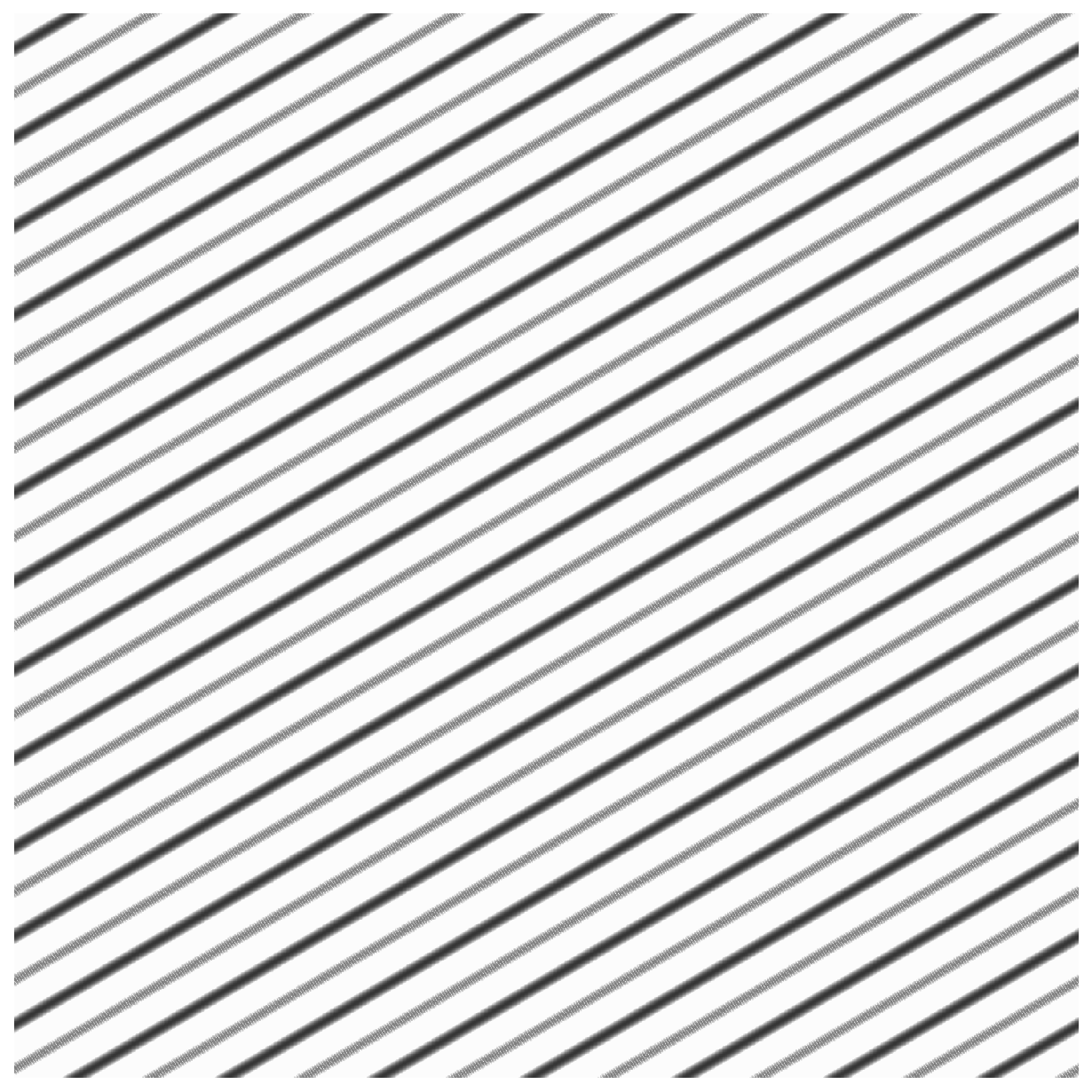} \\
$j=3$&
\includegraphics[width=2.5cm]{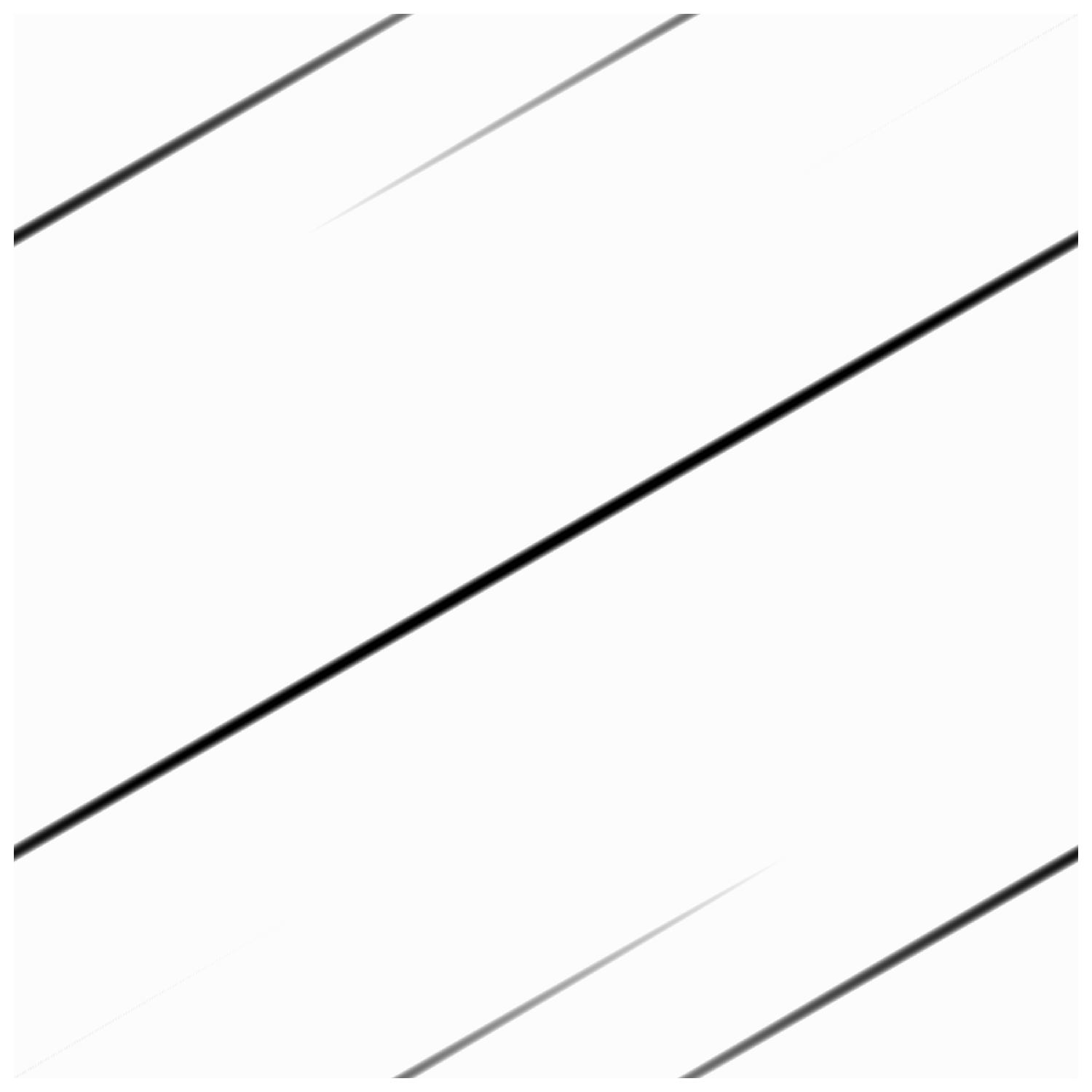}&
\includegraphics[width=2.5cm]{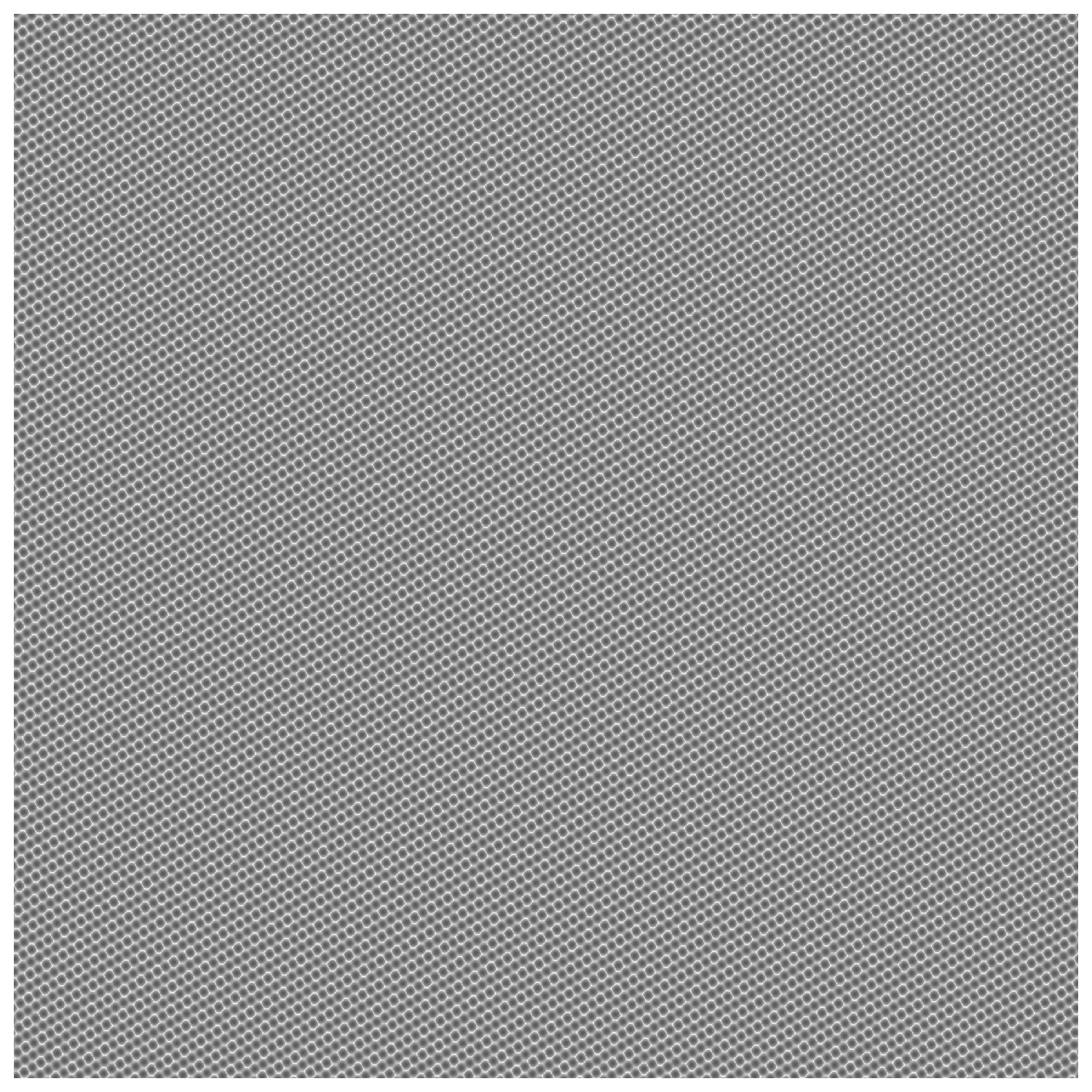} & $j=10$&
\includegraphics[width=2.5cm]{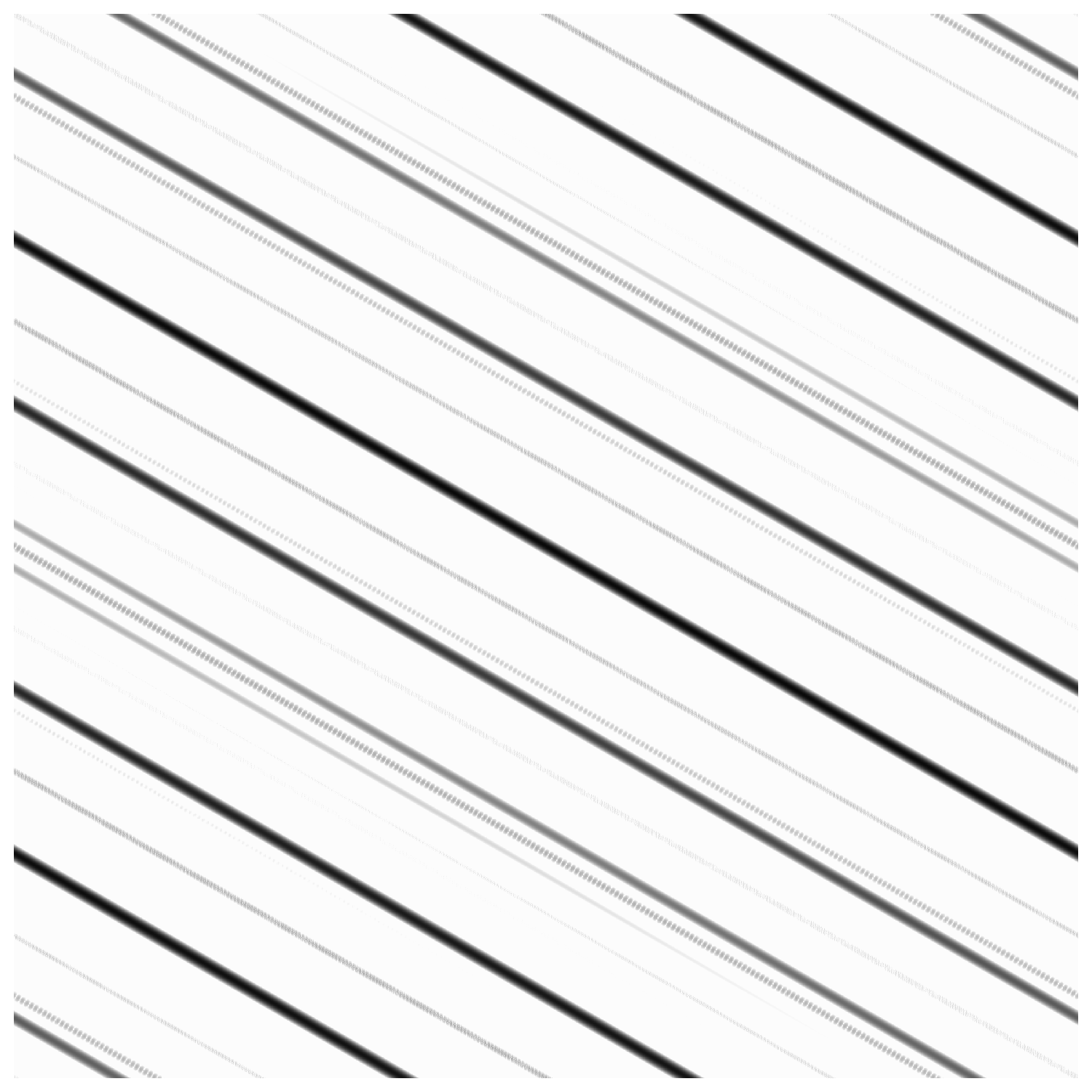}&
\includegraphics[width=2.5cm]{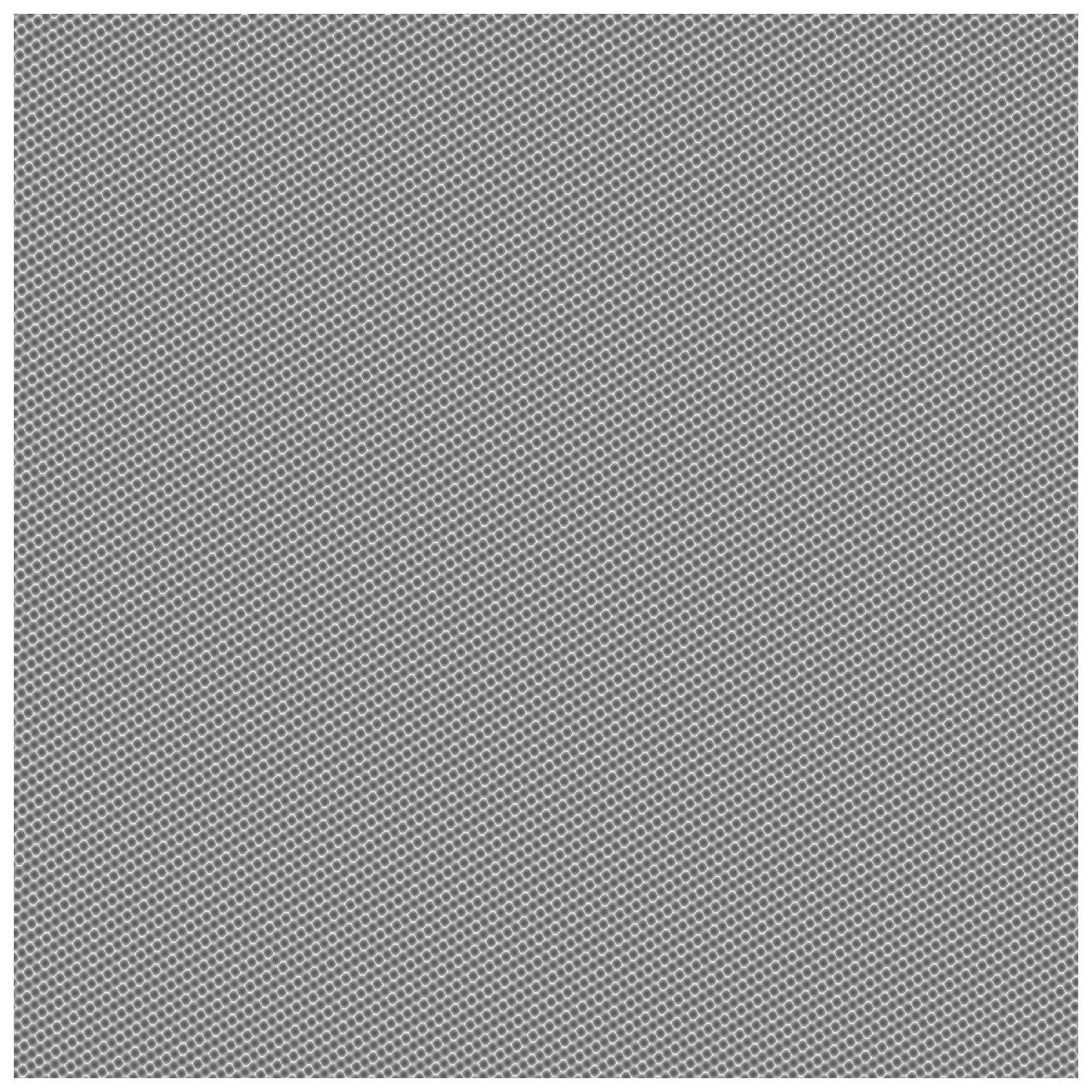} \\
$j=4$&
\includegraphics[width=2.5cm]{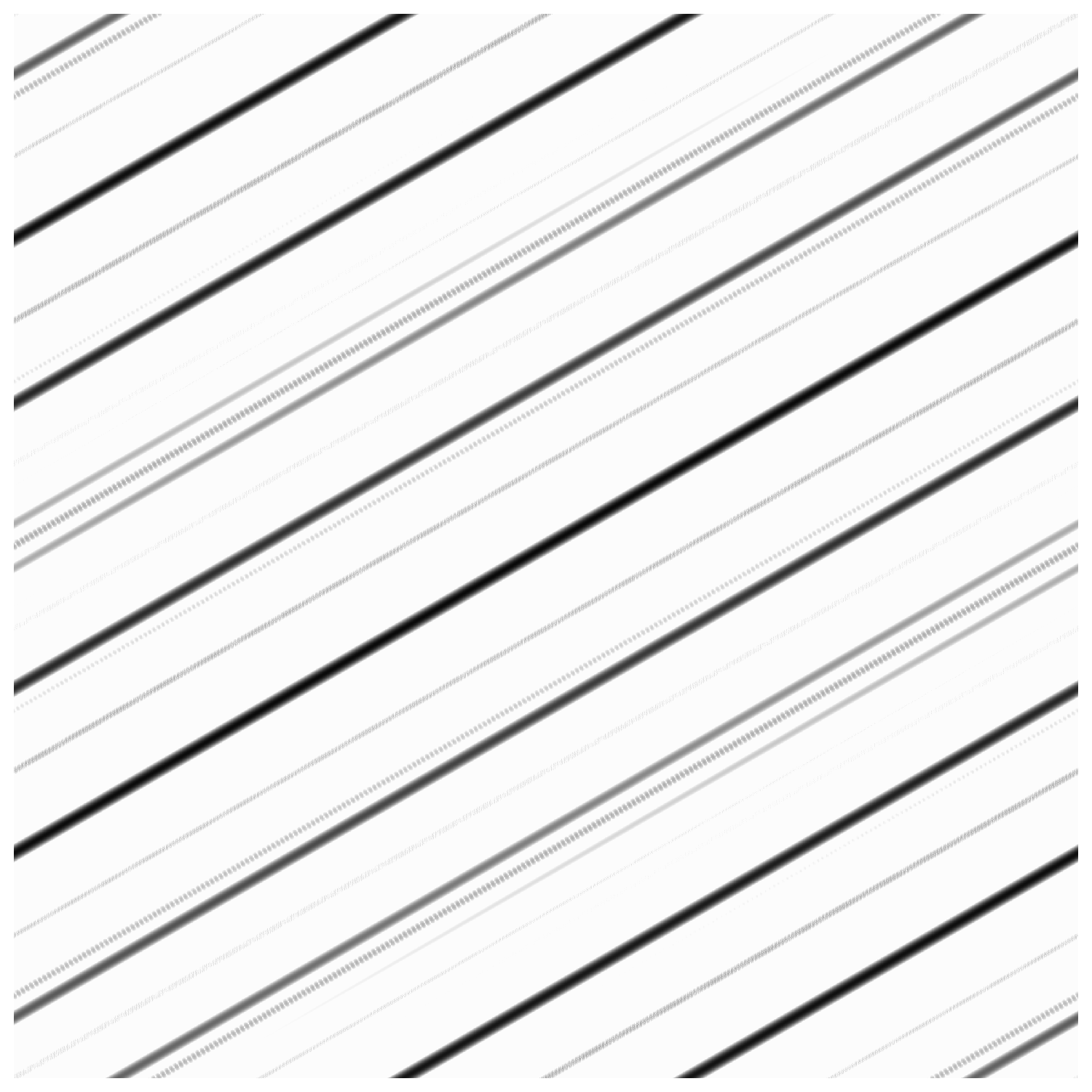}&
\includegraphics[width=2.5cm]{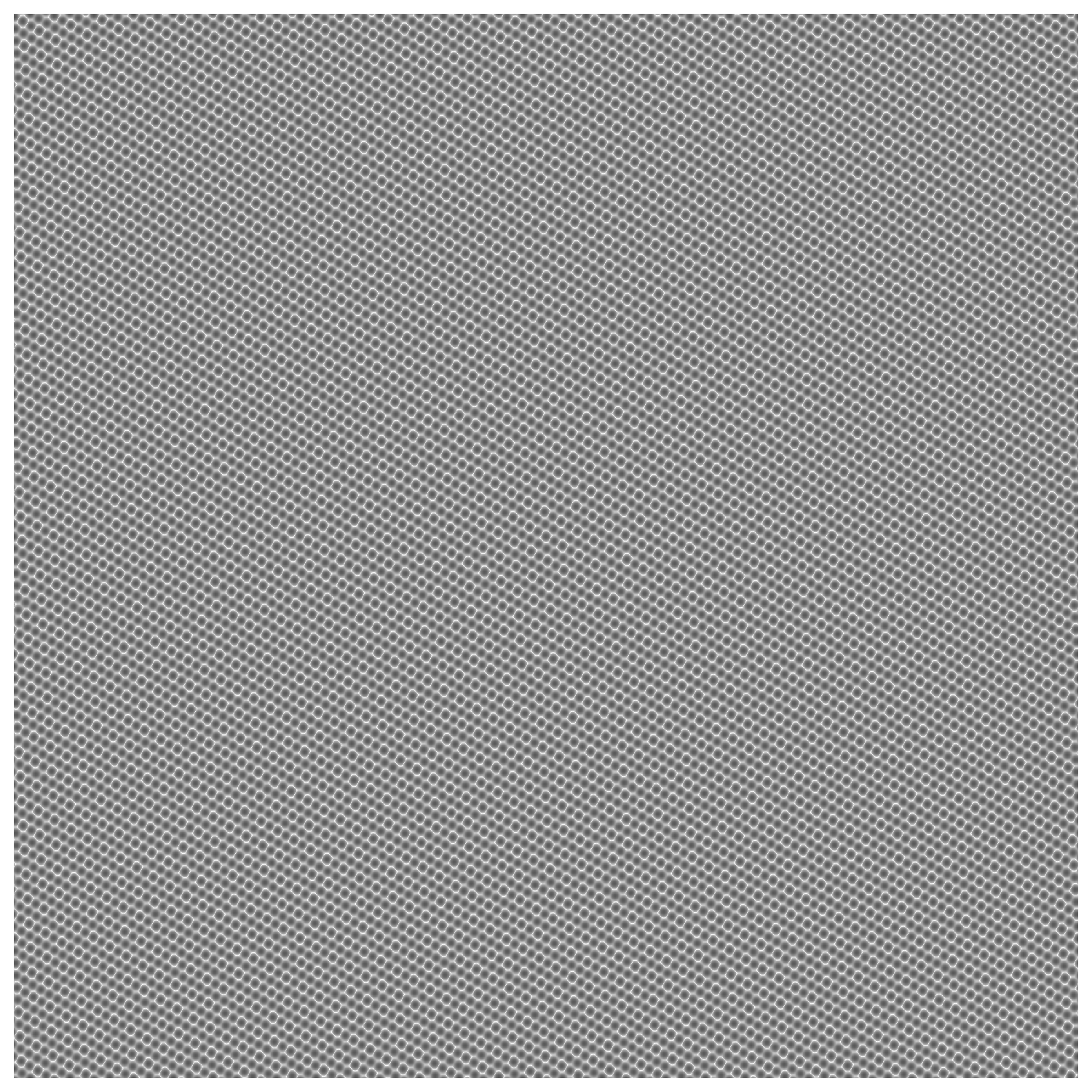} & $j=11$&
\includegraphics[width=2.5cm]{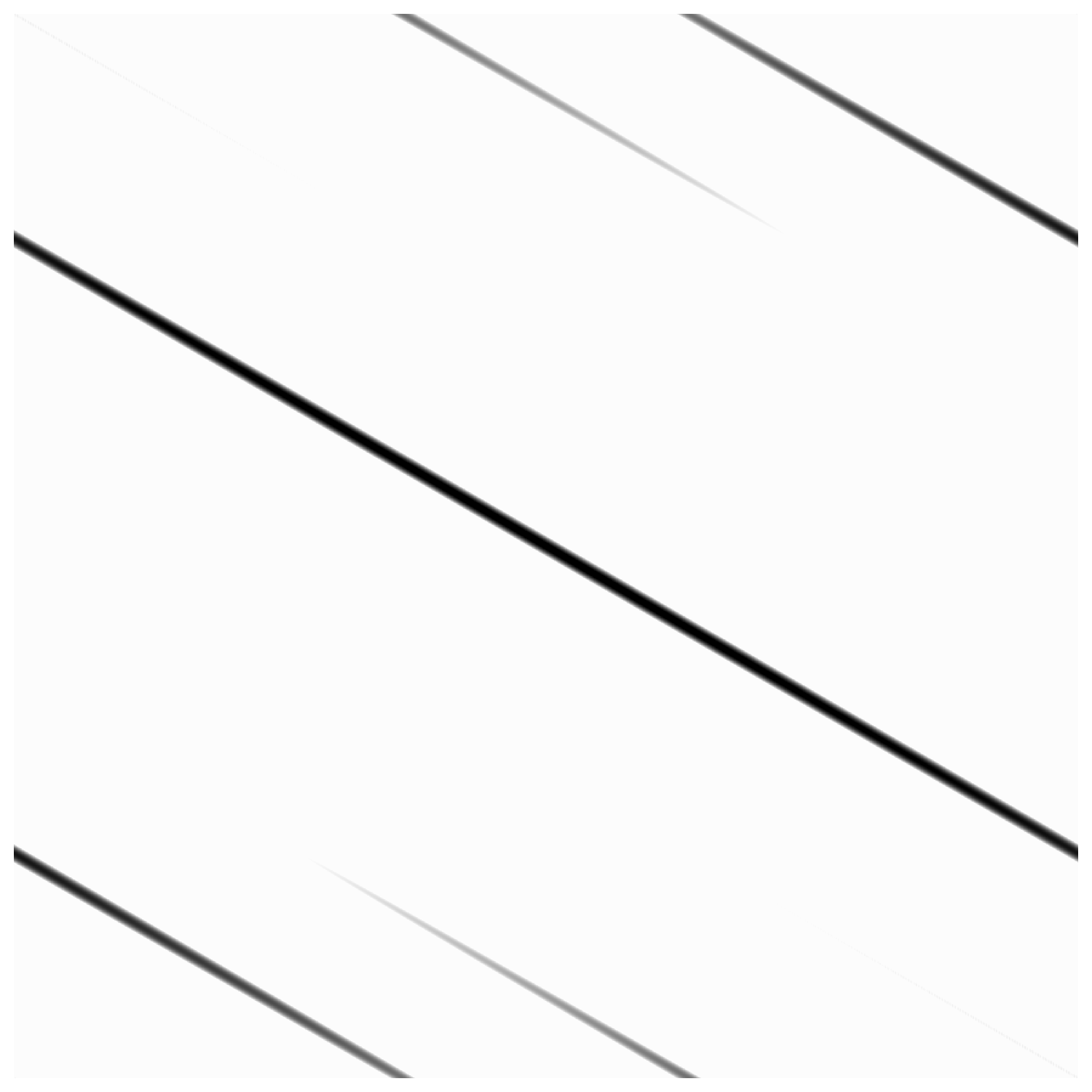}&
\includegraphics[width=2.5cm]{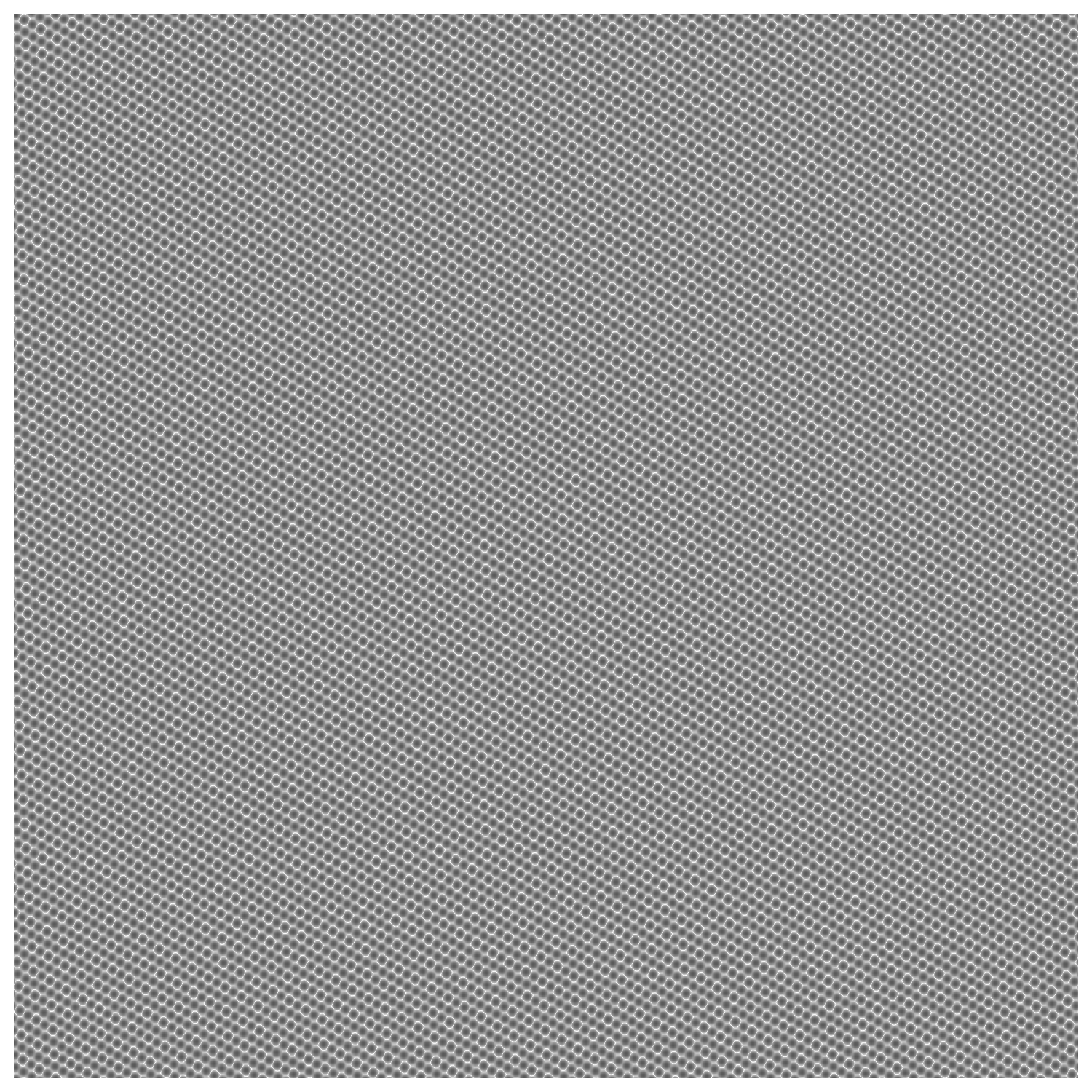} \\
$j=5$&
\includegraphics[width=2.5cm]{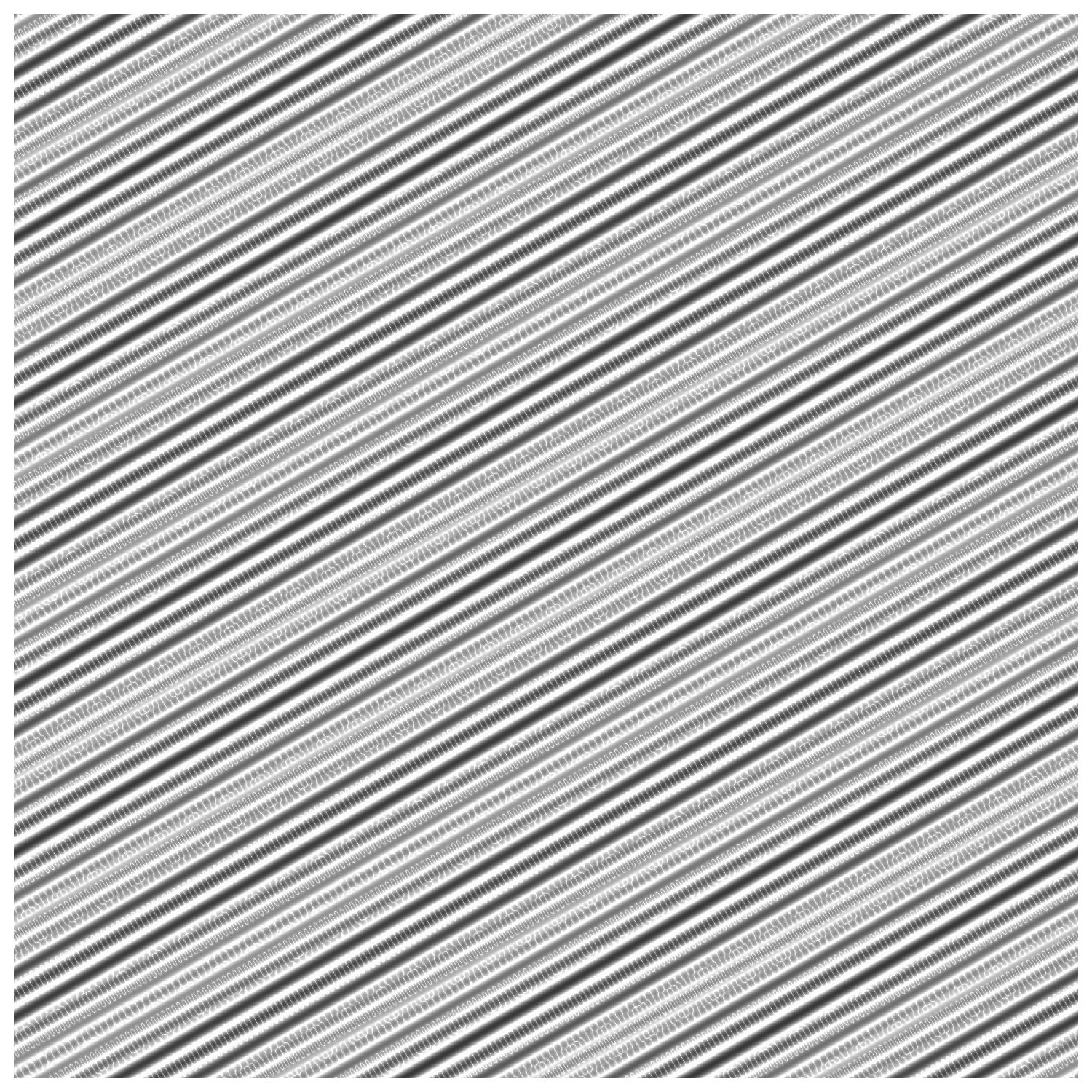}&
\includegraphics[width=2.5cm]{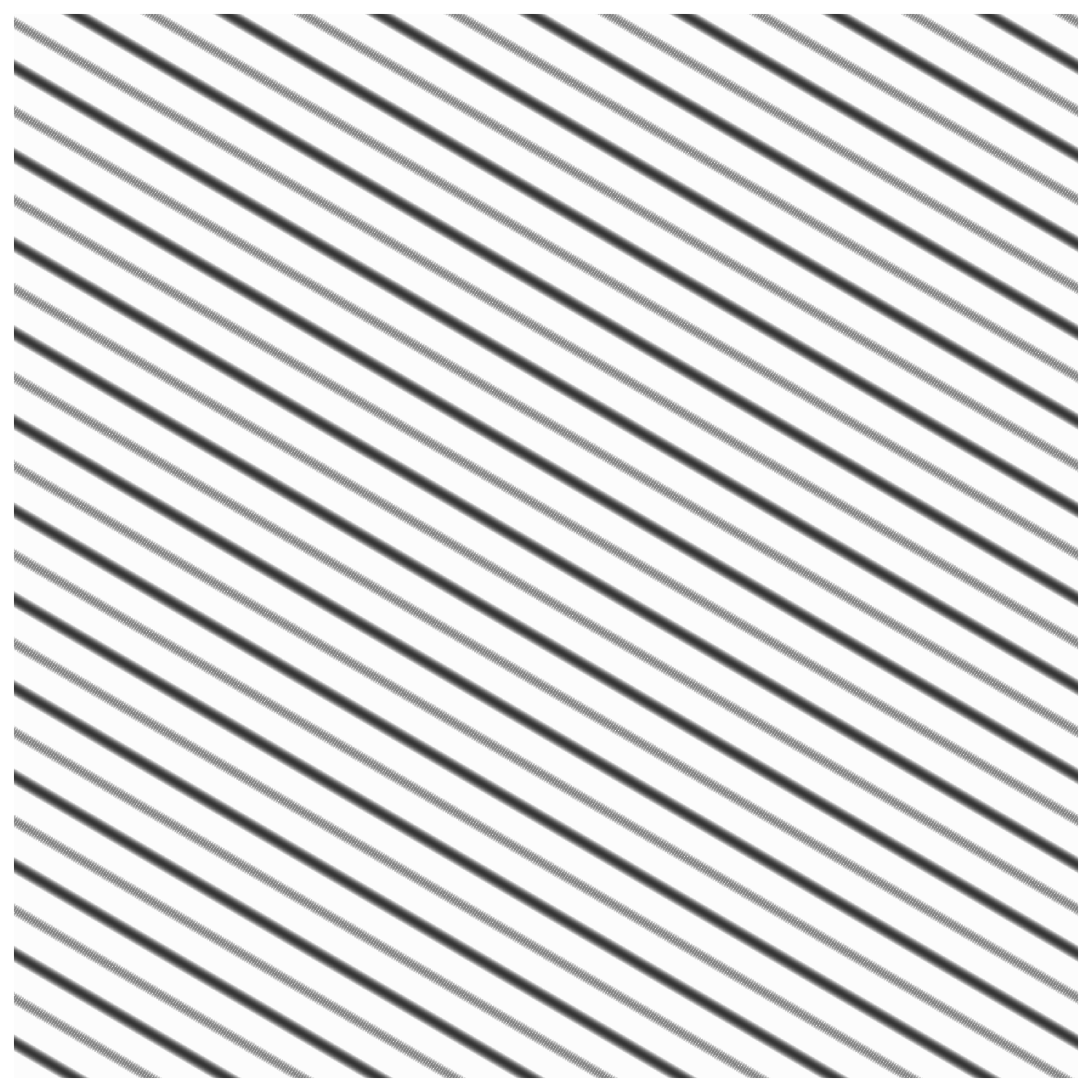} & $j=12$&
\includegraphics[width=2.5cm]{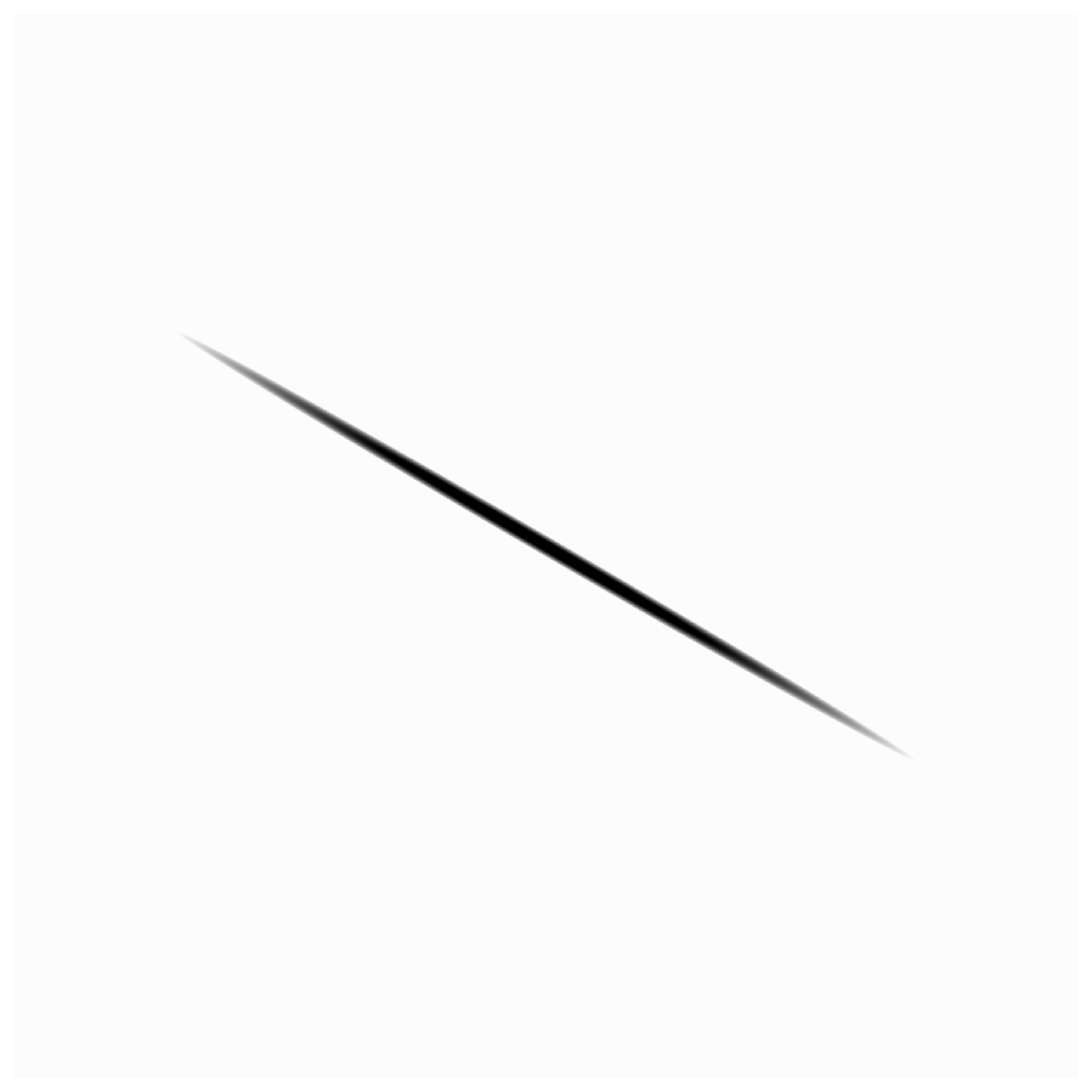}&
\includegraphics[width=2.5cm]{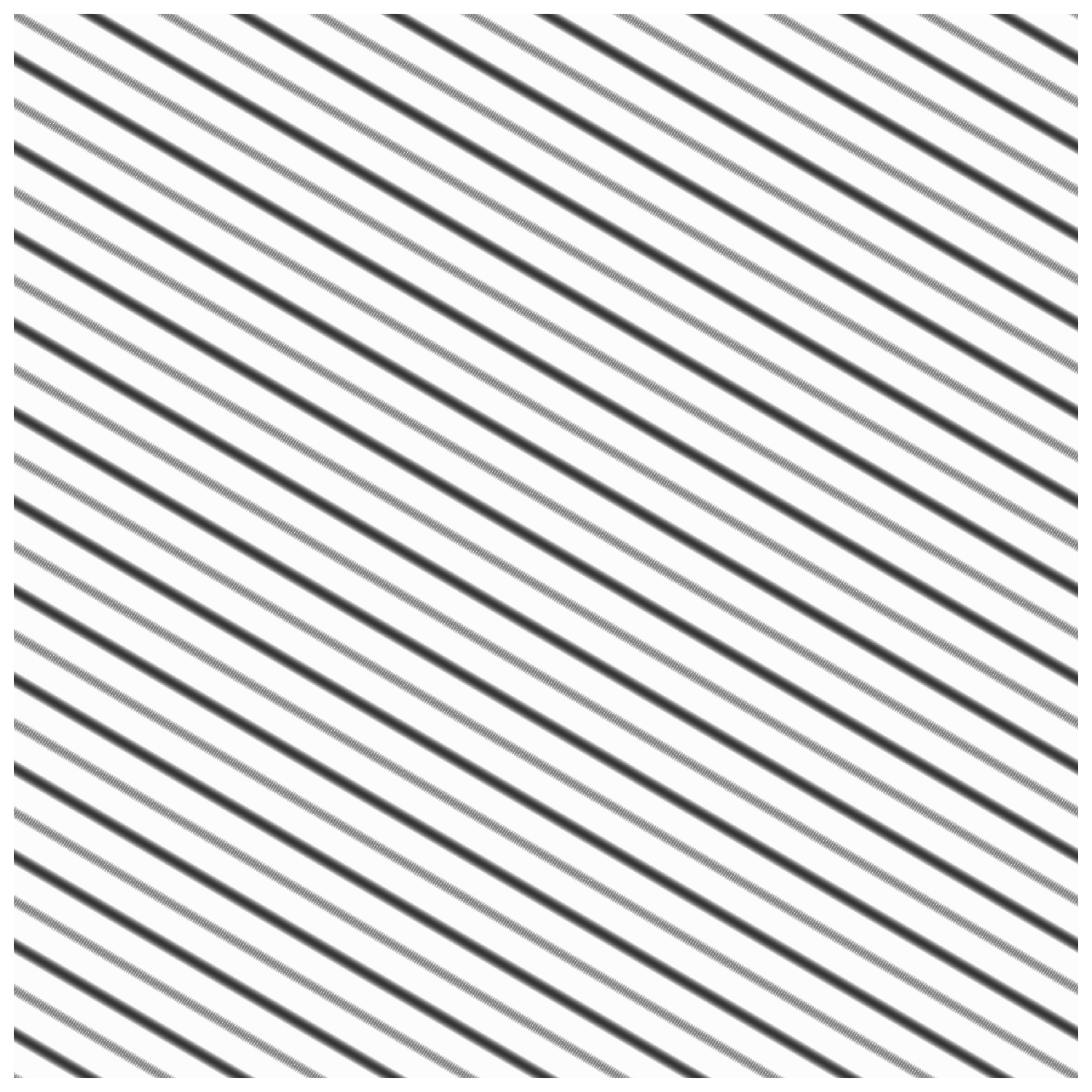} \\
$j=6$&
\includegraphics[width=2.5cm]{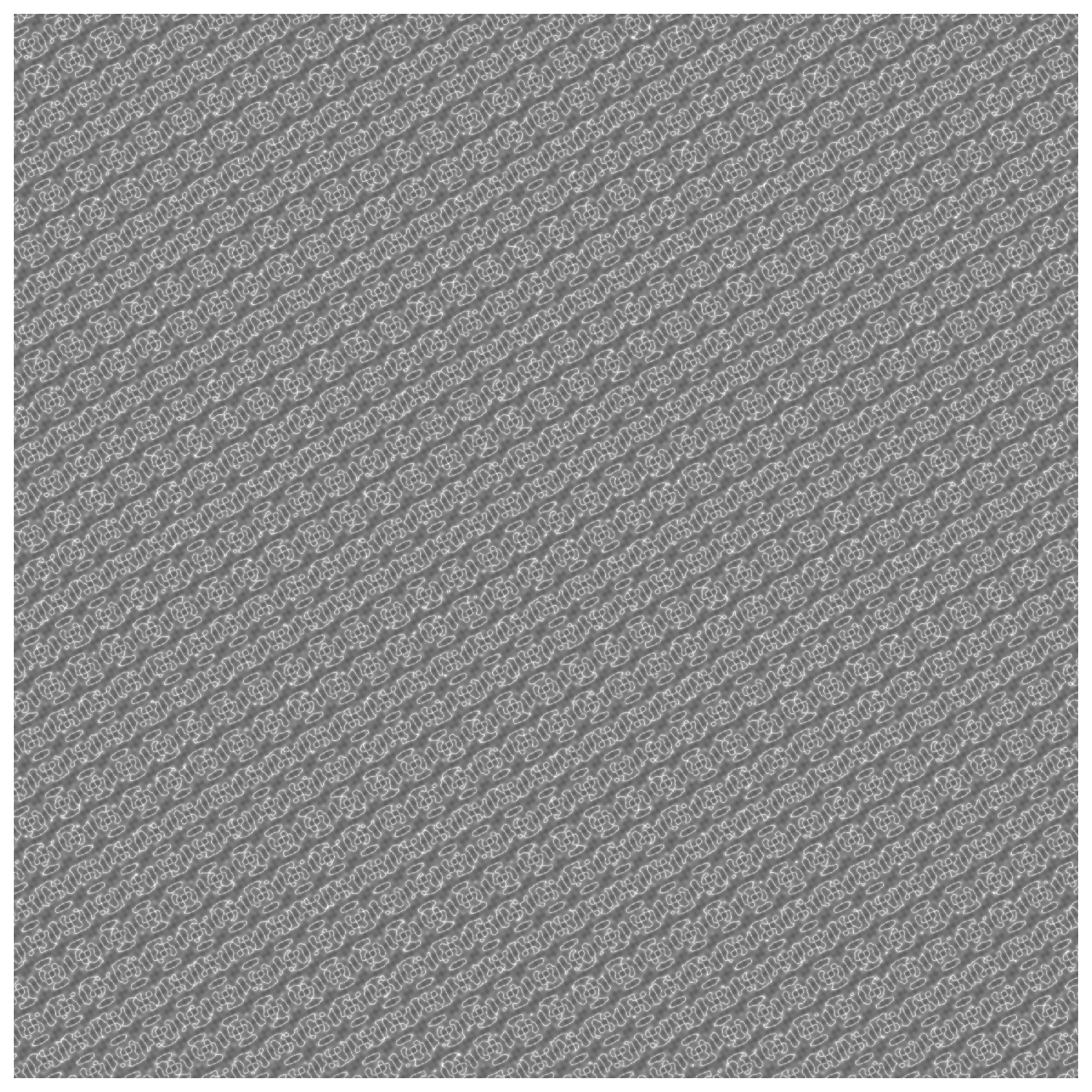}&
\includegraphics[width=2.5cm]{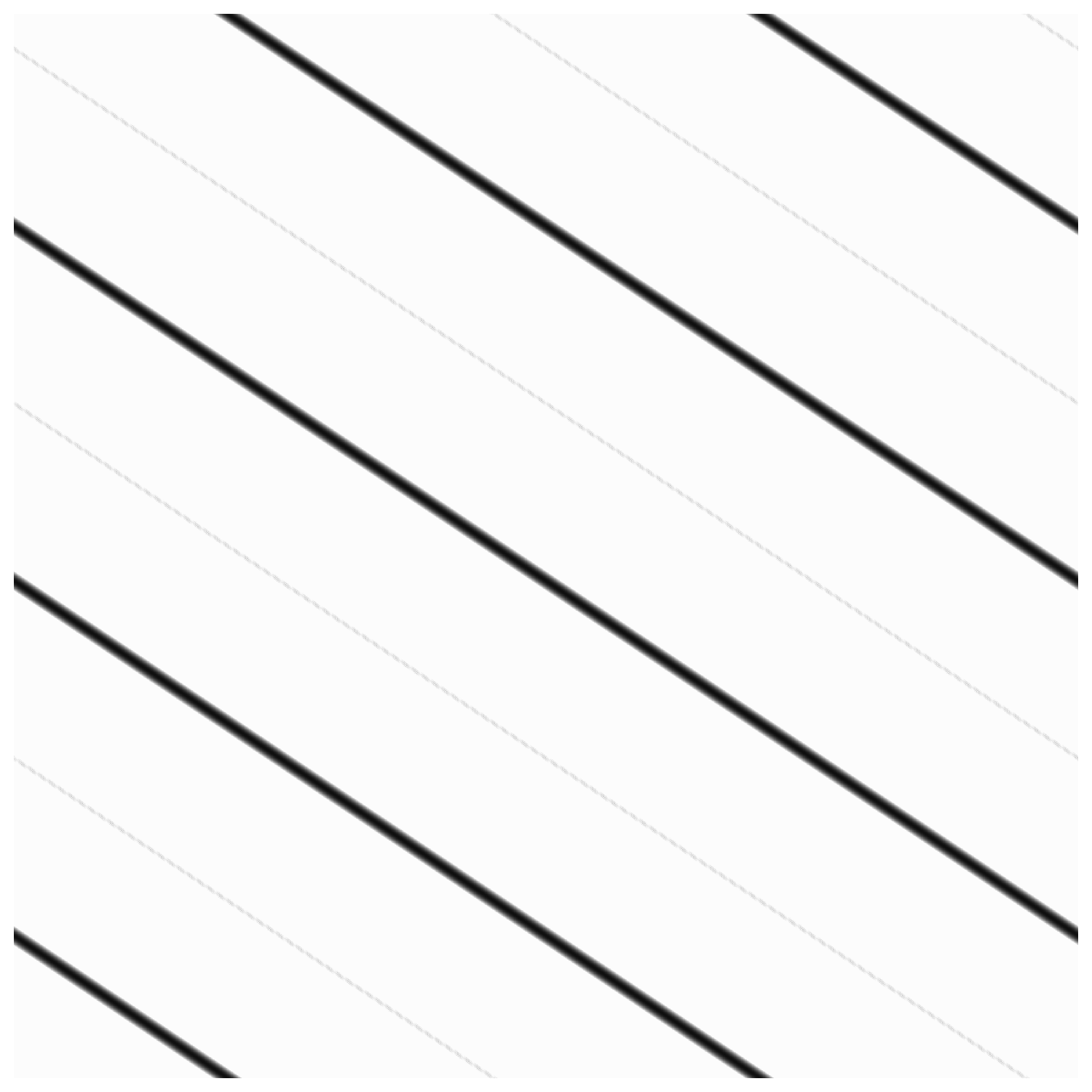} & $j=13$&
\includegraphics[width=2.5cm]{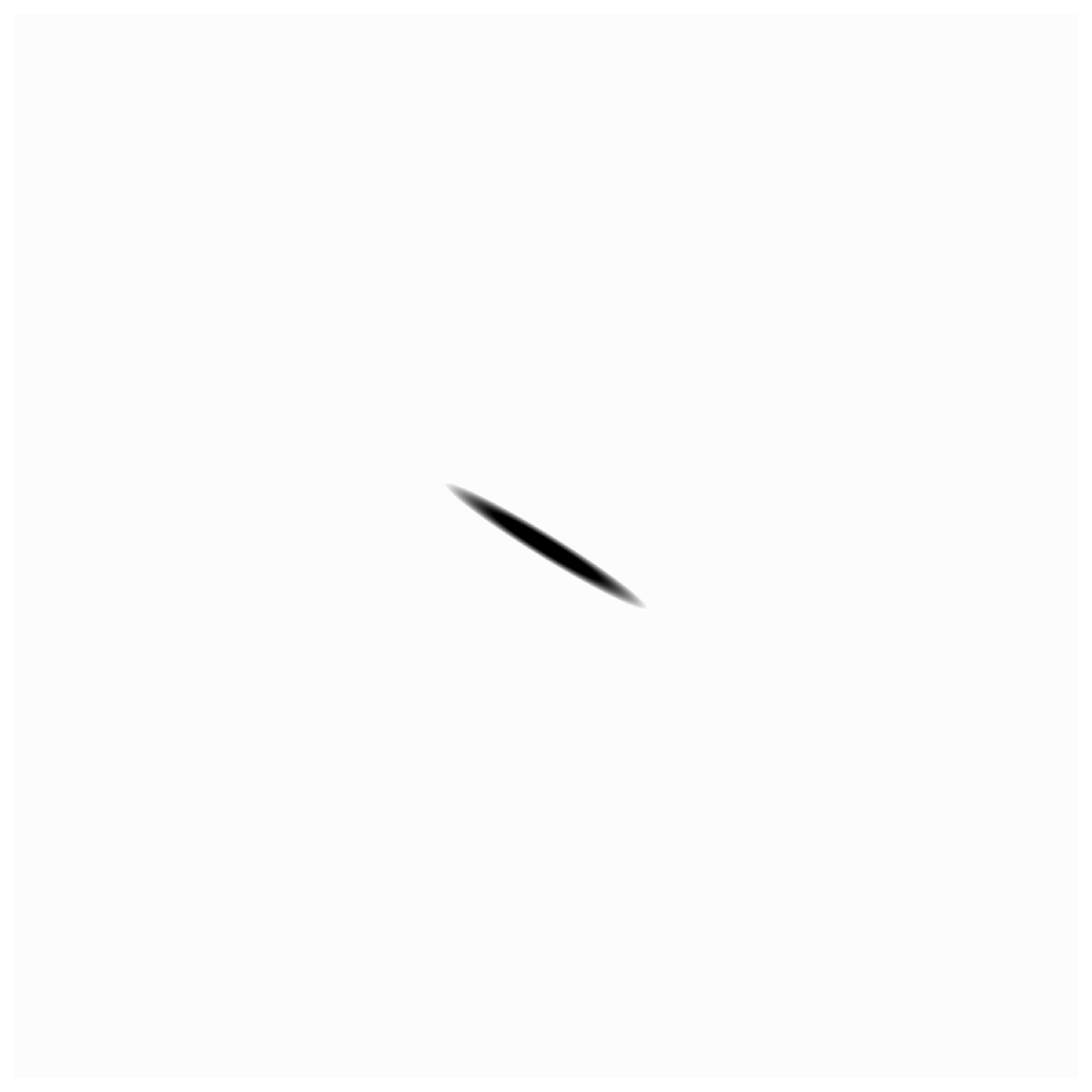}&
\includegraphics[width=2.5cm]{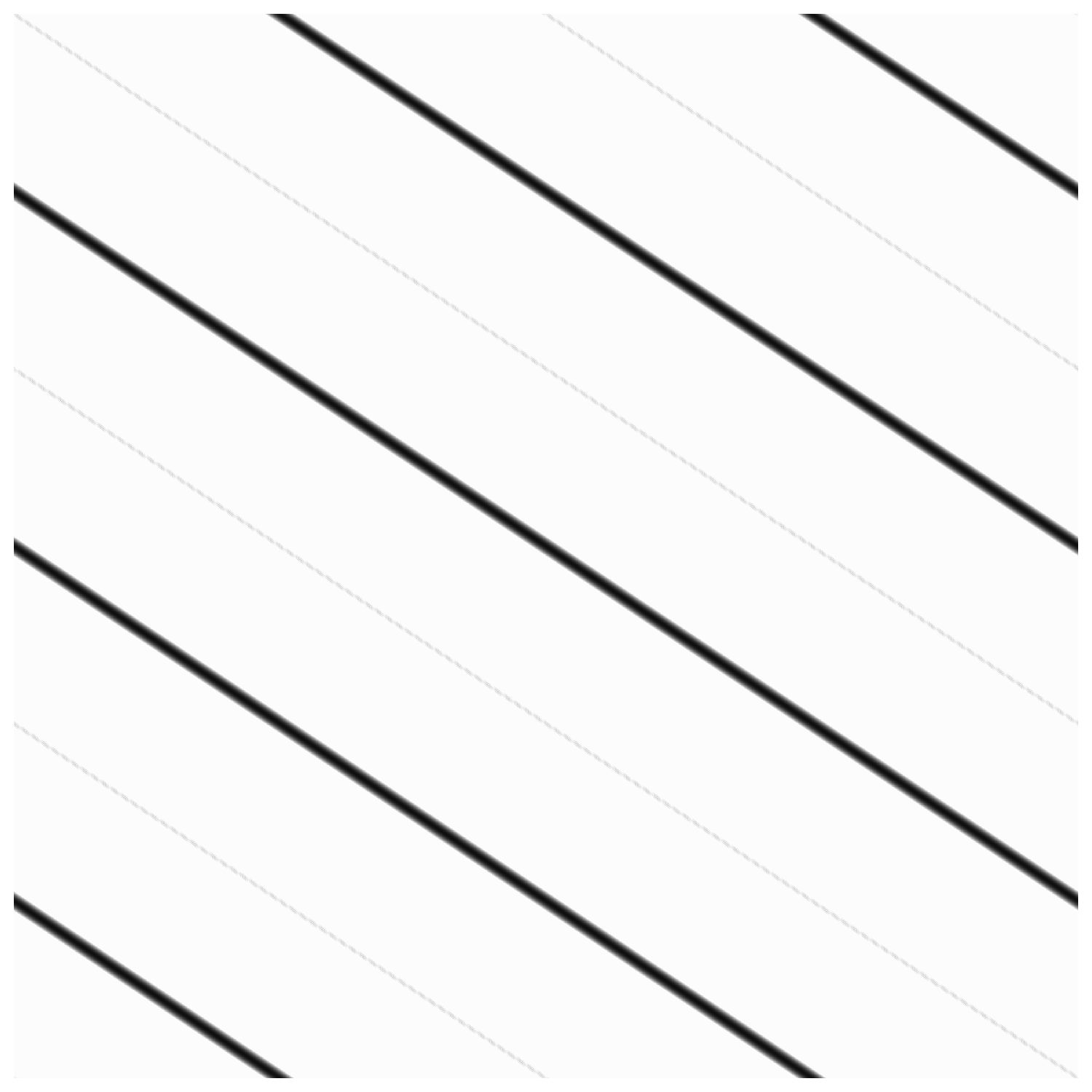} \\
\end{tabular}
\caption{In each of the two columns, the left image shows the Wigner function for $M^j g$. This particular Gaussian is centered at $(1/2, 1/2)$, a fixed point for $A$. The right images in each of the columns shows the Wigner function for $M^j e^0_{N/2}$.  Halfway through the period, $W_{M^je^0_{N/2}}(x)$ becomes less chaotic, something that does not happen for $W_{M^j g}(x)$. }\label{fig:wigner_j}
\end{figure}

Similar eigenfunctions were studied in \cite{scarred}. Specifically, the paper examines eigenfunctions of the form 
\begin{equation}\label{eq:gaussian}
\frac{1}{t_k} \sum_{t=0}^{t_k-1} e^{-i \frac{\phi_k + 2\pi s}{t_k}t}M^t g(x),
\end{equation}
where $g(x)$ is a Gaussian function. The limit of these eigenfuctions was shown to be half ergodic and half localized, in the sense of semiclassical measures.
Although 
\eqref{eq:delta} and \eqref{eq:gaussian} appear very similar, they actually exhibit distinct behavior, as demonstrated by Figure~\ref{fig:wigner_j} and Figure~\ref{fig:total_wigner}. In these figures, $A=\begin{bmatrix} 2 & 3 \\ 1 & 2 \end{bmatrix}$, $k=14$, and $N_k = 5822$.

\begin{figure}
   \centering
\begin{tabular}{lcc}
\includegraphics[width=6.5cm]{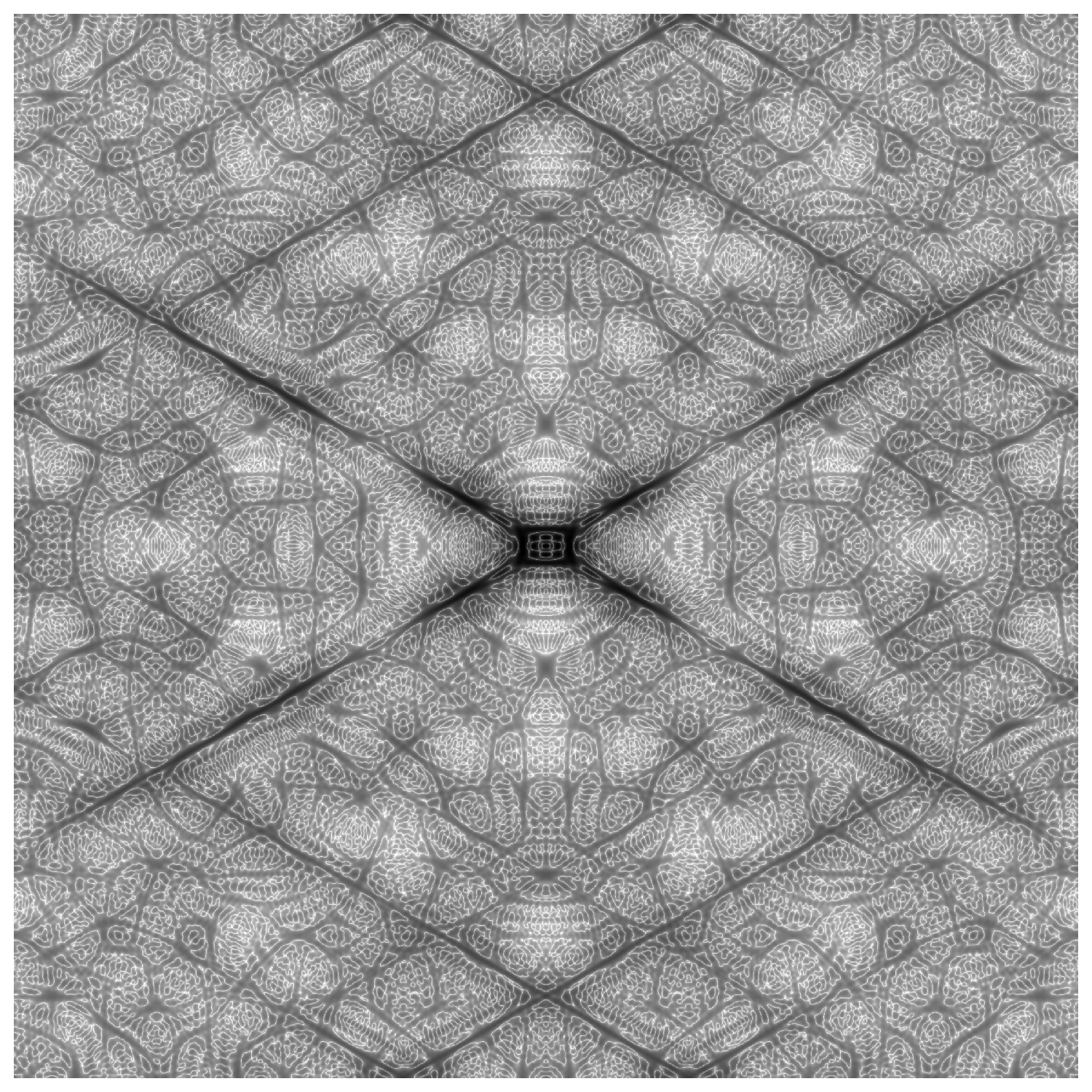}&
\includegraphics[width=6.5cm]{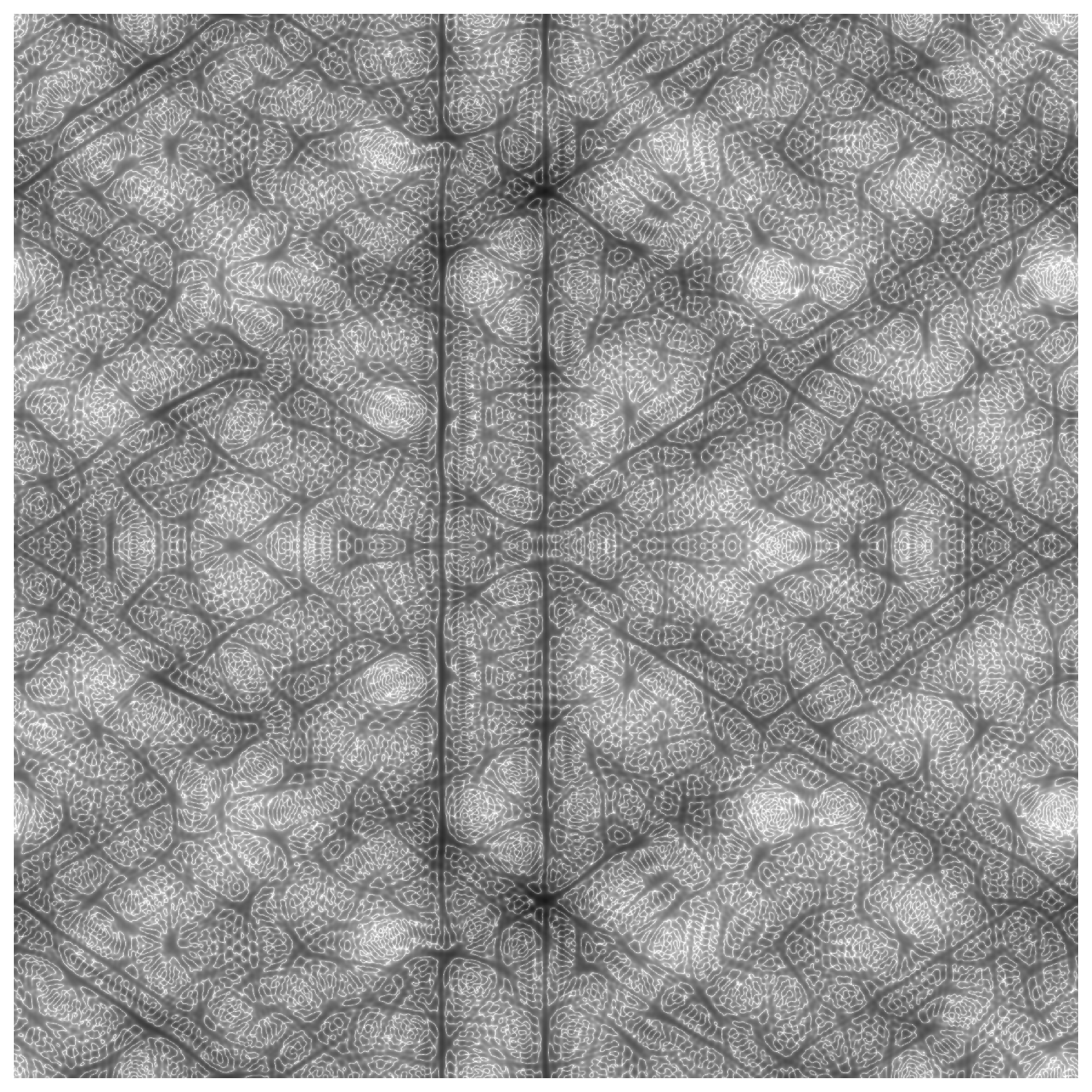} 
\end{tabular}
\caption{The left image shows the Wigner function for the eigenfunction \eqref{eq:gaussian}, where the Gaussian is centered at $(1/2, 1/2)$, a fixed point. The right image shows the Wigner function for the eigenfunction \eqref{eq:delta}.  Notice how the left image illustrates localization at $(1/2, 1/2)$, while the right image does not. We conjecture that the 
limit of eigenfunctions \eqref{eq:delta} equidistributes in the sense of semiclassical measures.}\label{fig:total_wigner}
\end{figure}

\section{Proof of Theorem \ref{thm:upperbound}}\label{secproofofupperbound}
We examine $A \in \SL(2, \Z)$ of the form (\ref{e:A-intro}). Note that the conditions on $A$ are less restrictive than those in  Section \ref{secproofoflowerbound}.
Let $N$ be odd.  We begin with the following dispersive estimate.
\begin{proposition}\label{propupperbound}
For $N$ odd and $M_{N, 0}: \cH_N(0) \rightarrow \cH_N(0)$, we have
\begin{align*}
    \left|\lrang{M_{N,0} e_j^0,e_k^0}_\mathcal{H} \right|\leq \frac{\sqrt{|b|}}{\sqrt{N}}.
\end{align*}
\end{proposition}

\begin{proof}
We begin by calculating $M_{N, 0} e_j^0$. By Lemma~\ref{lem:explicitformula}, for $\Phi(x, y) = \frac{d}{2b} x^2 - \frac{xy}{b} + \frac{a}{2b} y^2$,

$$M_{N, 0} e^0_j(x) = \frac{1}{\sqrt{|b|}} \int_\R e^{2 \pi iN \Phi(x, y)} \sum_{k \in \Z} \delta \left(y- \frac{Nk+j}{N} \right) dy=\frac{1}{\sqrt{|b|}} \sum_{k \in \Z} e^{2 \pi i N \Phi\left(x, \frac{Nk+j}{N}\right)}.$$

Thus, we focus on calculating $\Phi\left(x, \frac{Nk+j}{N}\right)$ mod $N\Z$.  In the following, we set $k = mb + r$ for $0 \leq r \leq |b| -1$.  We have
$$\Phi\left(x, \frac{Nk +j}{N} \right) =\frac{d}{2b}x^2 - xm - \frac{xr}{b} - \frac{xj}{Nb}+  \frac{ar^2}{2b}  + \frac{arj}{bN} + \frac{aj^2}{2bN^2} \mod N\Z.$$

Thus, using the Poisson summation formula and the fact that $\delta(x-x_0) f(x) = \delta(x-x_0) f(x_0)$, we know
\begin{align*}
M_{N,0} e_j^0(x) &= \frac{1}{\sqrt{|b|}}\sum_{r=0}^{|b|-1} e^{\frac{2 \pi i}{b} \left(  \frac{aNr^2}{2}  + arj + \frac{aj^2}{2N}\right)}  \sum_{m \in \Z}  e^{\frac{2 \pi i N}{b} \left(\frac{d}{2} x^2 -xr- \frac{xj}{N}\right)} e^{-2 \pi i N xm}\\
&=\frac{1}{N\sqrt{|b|}} \sum_{l \in \Z} \sum_{r=0}^{|b|-1} e^{\frac{2 \pi i}{b} \left(  \frac{d l^2}{2N} -lr - \frac{lj}{N} + \frac{ar^2N}{2} + arj + \frac{aj^2}{2N}\right)}   \delta\left(x - \frac{l}{N}\right).
\end{align*}

Now setting $$c_l = \sum_{r=0}^{|b|-1} e^{\frac{2 \pi i}{b} \left(\frac{d l^2}{2N} -lr - \frac{lj}{N} + \frac{ar^2N}{2} + arj + \frac{aj^2}{2N}  \right)},$$ we want to show that $c_l$ has period $N$, in other words $c_l =c_{l+N}$.
We know that$$c_{l+N} = e^{\frac{2 \pi i}{b} \left(\frac{d l^2}{2N} +dl + \frac{dN}{2} - \frac{lj}{N} -j + \frac{aj^2}{2N}  \right)}\sum_{r=0}^{|b|-1} e^{\frac{2 \pi i}{b} \left(-lr  -Nr+ \frac{ar^2N}{2} + arj \right)} =\alpha_{l+N} \sum_{r=0}^{|b|-1} \beta_{l+N, r}.$$

Recall in Section \ref{preliminaries}, we showed in order for $M_{N, 0}$ to descend  $\cH_N(0)$ to itself, $ab$ must even. Therefore, $$\beta_{l, r+b} = e^{\frac{2 \pi i}{b} \left( -lr -lb + \frac{ar^2 N}{2} +abrN + \frac{ab^2 N}{2}+ arj + abj \right)}=e^{\frac{2 \pi i}{b} \left( -lr + \frac{ar^2 N}{2} + arj  \right)}=\beta_{l,r}.$$
Therefore, $c_{l+N} = \alpha_{l+N} \sum_{r=d}^{|b|-1+d} \beta_{l+N, r}$.
As $ad-bc=1$, we know  $ad= 1$ mod $b$. Then,
$$c_{l+N} = \alpha_{l+N} \sum_{r=d}^{|b|-1+d} e^{\frac{2 \pi i}{b} \left(-lr  -Nr+ \frac{ar^2N}{2} + arj \right)} = \alpha_{l+N} \sum_{r=0}^{|b|-1} e^{\frac{2 \pi i}{b} \left(-lr  -ld-Nd+ \frac{ar^2N}{2} + \frac{dN}{2} + arj +j \right)}=c_l.$$

Using the fact that $c_k = c_{k+lN}$ for all $l \in \Z$ and that $\{e_j^0\}$ is an orthonormal basis, we have
\begin{align*}
&\lrang{M_{N,0} e_j^0(x), e_k^0(x)}_{\cH}\\
&= \frac{1}{\sqrt{N|b|}}   \lrang{\frac{1}{\sqrt{N}} \sum_{l \in \Z} c_{k+lN} \delta\left(x - \frac{Nl +k}{N}\right),  \frac{1}{\sqrt{N}} \sum_{l \in \Z} \delta\left(x- \frac{Nl + k}{N} \right)}_\cH\\
&= \frac{1}{\sqrt{N|b|}} \sum_{r=0}^{|b|-1}  e^{\frac{2 \pi i}{b} \left(  \frac{ar^2N}{2}  + arj + \frac{aj^2}{2N} + \frac{dk^2}{2N}- kr - \frac{kj}{N} \right)}.
\end{align*}

Using the triangle inequality, we conclude $|\lrang{M_{N,0} e_j^0(x), e_k^0(x)}_{\cH}| \leq \frac{\sqrt{|b|}}{\sqrt{N}}$. 
\end{proof}

\begin{figure}
    \centering
\includegraphics[scale=.39]{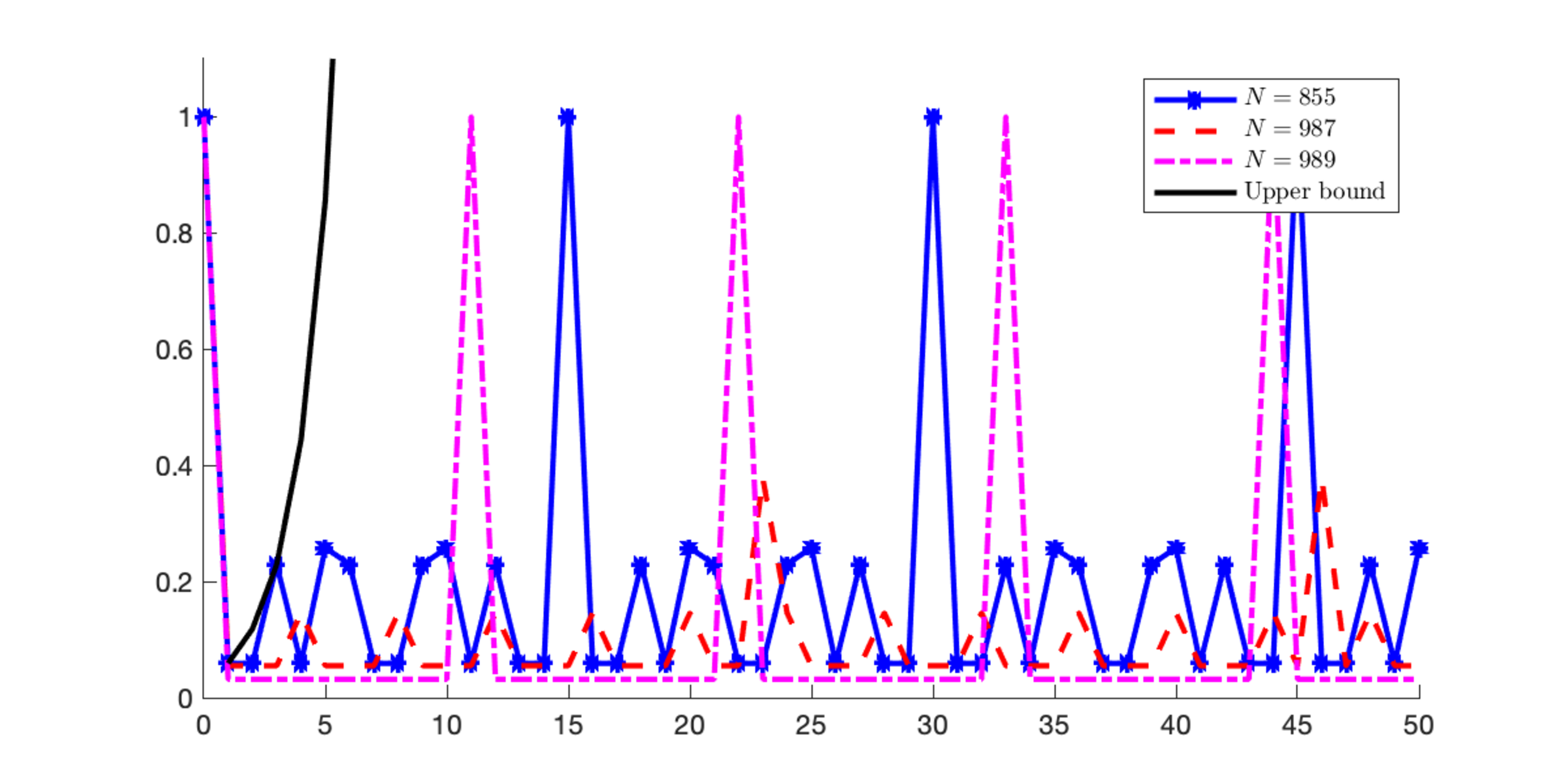}
    \caption{The plot of $\|M_{N,0}^j\|_{\ell^1 \rightarrow \ell^\infty}$ for $0 \leq j \leq 50$ and several values of $N$. $M_{N,0}$ corresponds to $A=\begin{bmatrix} 2 & 3 \\ 1 &2  \end{bmatrix}$. We also plot the upper bound $\sqrt{|b|/N}$ for $N=855$.}
    \label{fig:cat-prop}
\end{figure}

For $j>0$, we use the notation $A^j = \begin{bmatrix} a_j & b_j \\ c_j & d_j \end{bmatrix}$. As a direct consequence of Proposition \ref{propupperbound}, for $j>0$, $$\left\|M_{N,0}^j\right\|_{\ell^1 \rightarrow \ell^\infty}  \leq \frac{\sqrt{|b_j|}}{\sqrt{N}}.$$ Figure \ref{fig:cat-prop} compares this upper bound to actual values of $\|M_{N,0}^j\|_{\ell^1 \rightarrow \ell^\infty}$. Now as $\Tr A>2$, $A$ has eigenvalues $\lambda$, $\lambda^{-1}$ with $\lambda>1$. Then we have $|b_j| \sim \lambda^j$, giving  $$\left\|M_{N,0}^j\right\|_{\ell^1 \rightarrow \ell^\infty}  \leq \frac{C \lambda^{j/2}}{\sqrt{N}}.$$ 
Suppose $u$ is an eigenfunction of $M_{N, 0}$ with eigenvalue $\mu$. As  $M_{N, 0}$  is unitary, $|\mu| =1$. We have $u = \mu^{-n} M_{N,0}^n u$, which gives $u = \frac{1}{T} \left(\sum_{n=0}^{T-1} \mu^{-n} M^n_{N, 0} \right) u$.

Setting $B= \frac{1}{T} \sum_{n=0}^{T-1} \mu^{-n} M_{N,0}^n$, we see $u=Bu$. Again, as $M_{N, 0}$ is unitary,  $$B^* B= \frac{1}{T^2} \sum_{m,n=0}^{T-1} \mu^{m-n} M^{n-m}_{N,0}.$$ 
Then 
\begin{align*}
\|B^* B\|_{\ell^1 \rightarrow \ell^\infty}&\leq \frac{1}{T} + \frac{1}{T^2} \sum_{\substack{m \neq n \\ 0 \leq m,n \leq T-1}} \|M^{n-m}_{N,0} \|_{\ell^1 \rightarrow \ell^\infty}\\
&\leq \frac{1}{T} + \frac{C}{T^2} \sum_{\substack{m \neq n \\ 0 \leq m,n \leq T-1}} \frac{\lambda^{\frac{n-m}{2}}}{\sqrt{N}}\\
&\leq \frac{1}{T} + \frac{C}{\sqrt{N}}\lambda^{\frac{T}{2}}.
\end{align*}

We set $T= (1-\varepsilon/2)\log_\lambda N$ to get 
$$
\|B^* B\|_{\ell^1 \rightarrow \ell^\infty} \leq \frac{1}{(1-\varepsilon/2)\log_\lambda N} +\frac{C}{\sqrt{\lambda N^\varepsilon}}.
$$
 Finally, we know that $\|B\|^2_{\ell^2 \rightarrow \ell^\infty} = \|B^* B\|_{\ell^1 \rightarrow \ell^\infty}$. Therefore, for $0<\varepsilon<1$, there exists an $N_0$ such that for odd $N \geq N_0$, 
$$\|u\|_{\ell^\infty} \leq \|B\|_{\ell^2 \rightarrow \ell^\infty} \|u\|_{\ell^2} \leq \frac{1}{\sqrt{(1-\varepsilon)\log_\lambda N}}.$$


\begin{bibdiv}
\begin{biblist}

\bib{avakumovic}{article}{
  author = {V. G. Avakumovic},
  journal = {Mathematische Zeitschrift},
  pages = {327-344},
  title = {\"{U}ber die Eigenfunktionen auf geschlossenen Riemannschen Mannigfaltigkeiten},
  volume = {65},
  year = {1956},
}

\bib{berard}{article}{
  title = {On the Wave Equation on a Compact Riemannian Manifold without Conjugate Points},
  journal = {Mathematische Zeitschrift},
  volume = {155},
  pages = {249-276},
  year = {1977},
  author = {P. B\'{e}rard}
}


\bib{HB1980}{article}{
  title={Quantization of Linear Maps-Fresnel Diffraction by a Periodic Grating},
  author={ M.V. Berry},
  author={J.H. Hannay},
  volume={267},
  number={1},
  year={1980},
  publisher={Physica D}
}

\bib{Bonechi-DeBievre2000_Article_ExponentialMixingAndTimeScales}{article}{
  title={Exponential Mixing and {$|\log \hbar|$} Time Scales  in Quantized Hyperbolic Maps on the Torus},
  author={F. Bonechi},
  author={S. {De Bi\`{e}vre}},
  journal={Communications in Mathematical Physics},
  volume={211},
  number={3},
  pages={659--686},
  year={2000},
  publisher={Springer}
}

\bib{Bouzouina-deBievre}{article}{
  title={Equipartition of the Eigenfunctions of Quantized Ergodic Maps on the Torus},
  author={A. Bouzouina},
  author={S. {De Bi\`{e}vre}},
  volume={179},
  number={1},
  pages={83--105},
  year={1996},
  publisher={Communications in Mathematical Physics}
}

\bib{scarred}{article}{
  title={Scarred Eigenstates for Quantum Cat Maps of Minimal Periods},
  author={S. De Bi{\`e}vre},
  author={S. Nonnenmacher},
  author={F. Faure},
  journal={Commun. Math. Phys.},
  pages={449--492},
  volume={239},
  year={2003}
}

\bib{dyatlov2021semiclassical}{article}{
  title={Semiclassical Measures for Higher Dimensional Quantum Cat Maps},
  author={S. Dyatlov},
  author={M. J{\'e}z{\'e}quel},
  journal={Ann. Henri Poincar{\'e}},
  year={2023}
}


\bib{hormander}{article}{
  author = {L. H\"{o}rmander},
  journal = {Acta Math.},
  pages = {193 - 218},
  title = {The Spectral Function of an Elliptic Operator},
  volume = {121},
  year = {1968},
}



\bib{iwaniecsarnak}{article}{
  title = {$L^\infty$ Norms of Eigenfunctions of Arithmetic Surfaces},
  journal = {Ann. Math.},
  volume = {141},
  number = {2},
  pages = {301-320},
  year = {1995},
  author = {H. Iwaniec},
  author = {P. Sarnak}
}


\bib{Kurlberg-1}{article}{
  author = {P. Kurlberg},
  journal = {Ann. Henri Poincar\'{e}},
  pages = {75 - 89},
  title = {Bounds on Supremum Norms for Hecke Eigenfunctions of Quantized Cat Maps},
  volume = {8},
  year = {2007},
}


\bib{Kurlberg-Rudnick-0}{article}{
  author = {P. Kurlberg},
  author = {Z. Rudnick},
  journal = {Duke Math. J.},
  pages = {47-77},
  title = {Hecke Theory and Equidistribution for the Quantization of Linear Maps of the Torus},
  volume = {103},
  number = {1},
  year = {2000}
}
  


\bib{Kurlberg-Rudnick-1}{article}{
  title = {Value Distribution for Eigenfunctions of Desymmetrized Quantum Maps},
  journal = {International Mathematics Research Notices},
  author = {P. Kurlberg},
  author = {Z. Rudnick},
  volume = {2001},
  number = {18},
  pages = {985-1002},
  year = {2001},
}




\bib{levitan}{article}{
  author = {B. M. Levitan},
  journal = {Izv. Akad. Nauk SSSR Ser. Mat.},
  pages = {33–58},
  title = {On the Asymptotic Behavior of a Spectral Function and on Expansion in Eigenfunctions of a Self-Adjoint Differential Equation of Second Order},
  volume = {19},
  number = {1},
  year = {1955}
}


\bib{Olofsson-1}{article}{
  title = {Large Supremum Norms and Small Shannon Entropy for Hecke Eigenfunctions of Quantized Cat Maps},
  journal = {Communications in Mathematical Physics},
  author = {R. Olofsson},
  volume = {286},
  number = {3},
  pages = {1051-1072},
  year = {2009}
}
  
\bib{Olofsson-2}{article}{
  title = {Hecke Eigenfunctions of Quantized Cat Maps Modulo Prime Powers},
  journal = {Ann. Henri Poincar\'{e}},
  author = {R. Olofsson},
  volume = {1111},
  number = {10},
  year = {2009}
}


  
\bib{Olofsson-3}{article}{
  author = {R. Olofsson},
  journal = {Ann. Henri Poincar\'{e}},
  pages = {1285–1302},
  title = {Large Newforms of the Quantized Cat Map Revisited},
  volume = {11},
  year = {2010}
}

\bib{Rudnick-Sarnak}{article}{
  author = {Z. Rudnick},
  author = {P. Sarnak},
  journal = {Commun.Math. Phys.},
  pages = {195--213},
  title = {The behaviour of eigenstates of arithmetic hyperbolic manifolds},
  volume = {161},
  year = {1994}
}
 

\bib{z12semiclassical}{book}{
  title={Semiclassical Analysis},
  author={M. Zworski},
  volume={138},
  year={2012},
  publisher={American Mathematical Soc.}
}



\end{biblist}
\end{bibdiv}
\end{document}